\documentclass[11pt]{amsart}
\usepackage[letterpaper, total={6.5in,9in}, includefoot, centering]{geometry}
\usepackage{amsfonts}
\usepackage{amsmath,amsthm,amssymb,amscd,stmaryrd,mathtools}
\usepackage[foot]{amsaddr}
\usepackage[dvipsnames]{xcolor}

\usepackage[mathscr]{eucal}
\usepackage{latexsym}
\usepackage[new]{old-arrows}
\usepackage{graphics}
\usepackage{verbatim}
\usepackage{upgreek}
\usepackage{overpic}
\usepackage[all,cmtip]{xy} 
\usepackage{enumitem}
\usepackage{tikz}
\usepackage{tikz-cd}
\usepackage[pagebackref]{hyperref}
\renewcommand*\backref[1]{\ifx#1\relax \else (Page #1) \fi}
\usepackage[font=scriptsize]{caption}
\usepackage{upgreek}
\usepackage{algorithm}
\usepackage{algpseudocode}
\makeatletter
\newsavebox{\@brx}
\newcommand{\llangle}[1][]{\savebox{\@brx}{\(\m@th{#1\langle}\)}%
  \mathopen{\copy\@brx\kern-0.5\wd\@brx\usebox{\@brx}}}
\newcommand{\rrangle}[1][]{\savebox{\@brx}{\(\m@th{#1\rangle}\)}%
  \mathclose{\copy\@brx\kern-0.5\wd\@brx\usebox{\@brx}}}
\makeatother
\usepackage{float}
\graphicspath{ {./figures/} }

\usepackage[authoryear,sort&compress]{natbib}
\usepackage{adjustbox}
\bibpunct{[}{]}{;}{n}{,}{,}
\usepackage{cancel}
\usepackage{cleveref}
\hypersetup{
	colorlinks=true,
    linkcolor=red,
    citecolor=blue,
}

\makeatletter \@addtoreset{equation}{section}

\makeatother

\newtheorem{definition}{Definition}[section]
\newtheorem{remark}{Remark}[section]
\newtheorem{example}{Example}[section]
\newtheorem{prop}{Proposition}[section]

\newcommand{\bE}{{\vec{E}}}
\newcommand{\bT}{{\vec{T}}}
\newcommand{\bx}{\mathbf{x}}
\newcommand{\by}{\mathbf{y}}
\newcommand{\bu}{\mathbf{u}}
\newcommand{\bv}{\mathbf{v}}
\newcommand{\bp}{\mathbf{p}}
\newcommand{\bq}{\mathbf{q}}
\newcommand{\boldf}{\mathbf{f}}
\newcommand{\boldF}{\mathbf{F}}
\newcommand{\Em}{\textup{Em}}
\newcommand{\bxx}{{\mathbf{x}\mathbf{x}}}
\newcommand{\bxy}{{\mathbf{x}\mathbf{y}}}
\newcommand{\byx}{{\mathbf{y}\mathbf{x}}}
\newcommand{\byy}{{\mathbf{y}\mathbf{y}}}
\newcommand{\bL}{{\mathbb{L}}}

\usepackage{siunitx}
\DeclareSIUnit\str{str}
\DeclareSIUnit\sh{sh}
\DeclareSIUnit\erg{erg}
\begin{document}
\title{Preconditioning transformations of adjoint systems for evolution equations}
\author{Brian K. Tran$^{1,2,3}$, Ben S. Southworth$^2$, Hannah F. Blumhoefer$^4$ and Samuel Olivier$^5$}
\address{$^1$California State University Long Beach, Department of Mathematics and Statistics, Long Beach, CA 90840, USA}
\address{$^2$Los Alamos National Laboratory, Theoretical Division, Los Alamos, NM 87545, USA}
\address{$^3$University of Colorado Boulder, Department of Applied Mathematics, Boulder, CO 80309, USA}
\address{$^4$Iowa State University, Department of Aerospace Engineering, Ames, IA 50011, USA}
\address{$^5$Los Alamos National Laboratory, Computer, Computational and Statistical Sciences Division, Los Alamos, NM 87545, USA}
\email{brian.tran@csulb.edu, southworth@lanl.gov, hfb@iastate.edu, solivier@lanl.gov}
\allowdisplaybreaks

\begin{abstract}
    Achieving robust control and optimization in high-fidelity physics simulations is extremely challenging, especially for evolutionary systems whose solutions span vast scales across space, time, and physical variables. In conjunction with gradient-based methods, adjoint systems are widely used in the optimization of systems subject to differential equation constraints. In optimization, gradient-based methods are often transformed using suitable preconditioners to accelerate the convergence of the optimization algorithm. Inspired by preconditioned gradient descent methods, we introduce a framework for the preconditioning of adjoint systems associated to evolution equations, which allows one to reshape the dynamics of the adjoint system. We develop two classes of adjoint preconditioning transformations: those that transform both the state dynamics and the adjoint equation and those that transform only the adjoint equation while leaving the state dynamics invariant. Both classes of transformations have the flexibility to include generally nonlinear state-dependent transformations. Using techniques from symplectic geometry and Hamiltonian mechanics, we further show that these preconditioned adjoint systems preserve the property that the adjoint system backpropagates the derivative of an objective function. We then apply this framework to the setting of coupled evolution equations, where we develop a notion of scale preconditioning of the adjoint equations when the state dynamics exhibit large scale-separation. We demonstrate the proposed scale preconditioning on an inverse problem for the radiation diffusion equations. Naive gradient descent is unstable for any practical gradient descent step size, whereas our proposed scale-preconditioned adjoint descent converges in 10-15 gradient-based optimization iterations, with highly accurate reproduction of the wavefront at the final time.
\end{abstract}

\maketitle

{\hypersetup{linkcolor=black}\tableofcontents}

\section{Introduction}

Achieving robust control and optimization in high-fidelity physics simulations is extremely challenging, especially for evolutionary systems whose solutions span vast scales across space, time, and physical variables. Fundamental to optimization problems subject to differential equation constraints is the calculation of gradients with respect to some objective. Adjoint equations are ubiquitous in such problems, e.g. \cite{GiPi2000, Bloch2015, PiGi2000, NgGeBe2016, CaLiPeSe2003, Ca1981, Ro2005, EiLa2020, CaObTr2014}, which broadly take a gradient with respect to some objective at the final time and integrate this gradient \emph{backwards} in time using an adjoint linearization of the forward model to provide a gradient at the initial time. Particularly, the adjoint method enables efficient backpropagation of derivatives of objective functions when the dimension of the parameter space is significantly larger than the number of objectives \cite{Sa2016, MaMiYa2023}. When the backpropagation of derivatives through the adjoint equation is combined with a gradient-based optimization algorithm, this leads to methods for solving differential equation (DE)-constrained optimization problems.

Even in the setting of unconstrained optimization, naive gradient-based approaches can suffer slow convergence due to the poor shape of the optimization landscape. This is particularly true for problems with a vast range of solution scales across space, time, and variables (e.g., energy vs. temperature), which can cause the standard $\ell^2$-gradient to vanish or explode, making optimization intractably slow or unstable. In such cases, preconditioning can be used to reshape the derivative and accelerate convergence, e.g., \cite{precGD1, precGD2, precGD3, mirrordescent1, mirrordescent2, mirrordescent3, natgrad0, natgrad1}. The most well-known example is a Hessian preconditioner, which leads to Newton iterations, although such second order derivative information is typically not easily obtained for dynamic DEs.  Inspired by the broad concept of preconditioned gradient-based methods for unconstrained optimization, we develop a framework for preconditioning transformations of adjoint systems for evolution equations, in order to similarly reshape the derivative that is backpropagated through the adjoint equation. These \emph{adjoint preconditioning} transformations can be split into two classes: those that are induced from transformations of the state dynamics and those that transform only the adjoint equation while leaving the state dynamics invariant. 

Adjoint systems have an elegant description in terms of symplectic geometry and Hamiltonian dynamics. The study of the geometric nature of adjoint systems arose from the observation that the combination of any system of differential equations and its adjoint equations is described by a formal Lagrangian~\cite{Ib2006, Ib2007}. The symplectic structure of adjoint systems has been a valuable tool for understanding adjoint systems and their interaction with discretization; for example, in the context of Runge--Kutta methods, it was shown in \cite{Sa2016} that forming adjoints and discretization commute if and only if the discretization is symplectic. We further investigated the geometric structure of adjoint systems for ordinary differential equations (ODEs), differential-algebraic equations (DAEs), and evolutionary partial differential equations (PDEs) in \cite{TrLe2024, TrSoLe2024}, where symplectic geometry provided a natural lens to study properties of adjoint systems. In particular, the property that the adjoint equation backpropagates the derivative of an objective function is a consequence of the symplecticity of the Hamiltonian flow of the adjoint system \cite{TrLe2024}. As such, we will show that the adjoint preconditioning transformations we develop in this paper retain this backpropagation property by showing that these transformations are symplectomorphisms mapping the standard adjoint system into the preconditioned adjoint systems.

This paper is organized as follows. In \Cref{sec:geometry-intro}, we briefly review the symplectic geometry of adjoint systems (for a more thorough discussion, see \cite{TrLe2024, TrSoLe2024}). In \Cref{sec:adjoint-preconditioning-general}, we develop preconditioning transformations of adjoint systems. To begin, in \Cref{sec:duality}, we review notions of derivatives, gradients and duality, and discuss preconditioning of gradient descent methods for unconstrained optimization. Inspired by preconditioning of gradient descent, we subsequently introduce several analogous adjoint preconditioning transformations. Specifically, in \Cref{sec:induced-from-state}, we discuss adjoint preconditioning induced by a transformation of the state dynamics. In \Cref{sec:induced-from-duality}, we discuss adjoint preconditioning induced by a duality pairing and generalize this in \Cref{sec:fiberwise-duality} to consider fiberwise duality pairings. Interestingly, the adjoint preconditioning using a fiberwise duality pairing can be naturally expressed in terms of a vector bundle connection. Furthermore, each of these adjoint preconditioning transformations will be interpreted as symplectomorphisms that pull back the standard adjoint system. As an application of such adjoint preconditioning, we consider adjoint systems for coupled evolution equations in \Cref{sec:adjoint-systems-coupled} and, particularly, in \Cref{sec:scale-preconditoning-coupled}, introduce a notion of \emph{scale preconditioning} when the coupled evolution equations exhibit large scale-separation. We demonstrate the proposed scale preconditioning with a numerical example in \Cref{sec:numerical}, involving an inverse problem for a perturbation of a thick Marshak wave governed by the radiation diffusion equations. Naive gradient descent is unstable for any practical gradient descent step size, whereas our proposed scale-preconditioned adjoint descent converges in 10-15 gradient-based optimization iterations, with highly accurate reproduction of the wavefront at the final time.

\subsection{The geometry of adjoint systems}\label{sec:geometry-intro}
In \cite{TrLe2024}, we performed a study of the geometry of adjoint systems associated with ODEs and DAEs, utilizing tools from symplectic and presymplectic geometry, respectively. In \cite{TrSoLe2024}, we extended these results to the infinite-dimensional setting by considering adjoint systems associated with evolution equations on a Banach space. Particularly, symplectic geometry has proven to be a natural tool to analyze adjoint systems and the interplay between discretization and optimization. In this paper, we focus on preconditioning transformations for adjoint systems, which can be interpreted naturally from the perspective of symplectic geometry. We will begin with a brief review of the symplectic geometry associated with adjoint systems. For our purposes, we will consider state dynamics on a finite-dimensional vector space; see \cite{TrLe2024} and \cite{TrSoLe2024} for details on the more general manifold and infinite-dimensional settings.

Let $U$ be a finite-dimensional vector space and consider the ODE on $U$ given by
\begin{equation}\label{ODEmanifold}
\dot{\bu} = \boldf(\bu),
\end{equation} 
where $\boldf$ is a vector field on $U$. Letting $\pi: TU \rightarrow U$ denote the tangent bundle projection, we recall that a vector field $\boldf$ is a map $\boldf: U \rightarrow TU$ which satisfies $\pi \circ \boldf = \mathbf{1}_U$, i.e., $\boldf$ is a section of the tangent bundle. Note that since $U$ is a vector space, its tangent bundle is trivial $TU = U \times U$ and similarly, the cotangent bundle $T^*U = U \times U^*$. Thus, a vector field on $U$ is simply a map $\boldf: U \rightarrow U$.

Consider the \emph{adjoint Hamiltonian} $H: T^*U \rightarrow \mathbb{R}$ corresponding to \eqref{ODEmanifold}, given by 
\begin{equation}\label{eq:adjoint-hamiltonian-canonical}
    H(\bu,\bp) = \llangle \bp, \boldf(\bu) \rrangle_\bu,
\end{equation}
where $\llangle\cdot,\cdot\rrangle_\bu$ is the standard duality pairing of $T^*_\bu U$ with $T_\bu U$; we will often omit the basepoint $\bu$ in the pairing when it is clear in context. Recall that the cotangent bundle $T^*U$ possesses a canonical symplectic form $\Omega = -d\Theta$ where $\Theta$ is the tautological one-form on $T^*U$. With coordinates $(\bu,\bp) = (u^\alpha, p_\beta)$ on $T^*U$, the canonical symplectic form has the coordinate expression 
$$\Omega = - \llangle dp \wedge dq \rrangle := - \sum_\alpha dp_\alpha \wedge dq^\alpha.$$
We define the adjoint system as the ODE on $T^*U$ given by Hamilton's equations, with the above choice of Hamiltonian $H$ and the canonical symplectic form. Thus, the adjoint system is given by the equation
$$ i_{X_H}\Omega = dH, $$
whose solution curves on $T^*U$ are the integral curves of the Hamiltonian vector field $X_H$. As is well-known, for the adjoint Hamiltonian \eqref{eq:adjoint-hamiltonian-canonical}, its Hamiltonian vector field $X_H$ is given by the cotangent lift $\widehat{f}$ of $f$, which is a vector field on $T^*U$ that covers $f$ (for a discussion of the geometry of the cotangent bundle and lifts, see \cite{YaIs1973, MaRa1999}). This is stated in the following proposition.
\begin{prop}
    For the adjoint Hamiltonian \eqref{eq:adjoint-hamiltonian-canonical}, the corresponding Hamiltonian vector field $X_H$ is equal to the cotangent lift $\widehat{\boldf}$ of $\boldf$. 
\end{prop}
\begin{proof}
    To be more explicit, recall that the cotangent lift of $\boldf$ is constructed as follows. Let $\Phi_{\epsilon}: U \rightarrow U$ denote the time-$\epsilon$ flow map of $\boldf$, which is a one-parameter family of diffeomorphisms. Then, we consider the cotangent lifted diffeomorphisms given by $T^*\Phi_{-\epsilon}: T^*U \rightarrow T^*U$. This covers $\Phi_{\epsilon}$ in the sense that $\pi_{T^*U} \circ T^*\Phi_{-\epsilon} = \Phi_{\epsilon} \circ \pi_{T^*U} $ where $\pi_{T^*U}: T^*U \rightarrow U$ is the cotangent bundle projection. The cotangent lift $\widehat{\boldf}$ is then defined to be the infinitesimal generator of the cotangent lifted flow, for $z \in T^*U$,
    $$ \widehat{\boldf}(z) = \frac{d}{d\epsilon}\Big|_{\epsilon=0} T^*\Phi_{-\epsilon} (z). $$
We can directly verify that $\widehat{\boldf}$ is the Hamiltonian vector field for $H$, which follows from
$$ i_{\widehat{\boldf}}\Omega = -i_{\widehat{\boldf}}d\Theta = -\mathcal{L}_{\widehat{\boldf}}\Theta + d( i_{\widehat{\boldf}}\Theta ) = d( i_{\widehat{\boldf}}\Theta) = dH, $$
where $\mathcal{L}_{\hat{\boldf}}\Theta = 0$ follows from the fact that cotangent lifted flows preserve the tautological one-form and $H = i_{\widehat{\boldf}}\Theta$ follows from a direct computation, where $i_{\widehat{\boldf}}\Theta$ is interpreted as a real-valued function on the cotangent bundle which maps $(\bu,\bp)$ to $\llangle \Theta(\bu,\bp), \widehat{\boldf}(\bu,\bp)\rrangle_\bu$.
\end{proof}

With coordinates $z = (\bu,\bp)$ on $T^*U$, the adjoint system is the ODE on $T^*U$ given by
\begin{equation}\label{AdjointODEmanifold}
\dot{z} = \widehat{\boldf}(z). 
\end{equation}

The adjoint system \eqref{AdjointODEmanifold} covers \eqref{ODEmanifold} in the following sense.
\begin{prop}[Proposition 2.2 of \cite{TrLe2024}] \label{LiftIntegralCurveProp}
Integral curves to the adjoint system \eqref{AdjointODEmanifold} lift integral curves to the system \eqref{ODEmanifold}. \qed
\end{prop}

In coordinates, the adjoint system has the form
\begin{align*}
    \dot{\bu} &= \boldf(\bu), \\
    \dot{\bp} &= -[D\boldf(\bu)]^*\bp,
\end{align*}
where $D\boldf$ is the Fr\'{e}chet derivative of $\boldf$ (we will discuss this in more detail in \Cref{sec:duality}) and $[\ \cdot\ ]^*$ denotes the adjoint of $D\boldf(\bu)$ with respect to the duality pairing $\llangle\cdot,\cdot\rrangle_\bu$. We will refer to the variable $\bu$ as the state variable and $\dot{\bu} = \boldf(\bu)$ as the state dynamics; we refer to $\bp$ as the adjoint variable and $\dot{\bp} = -[D\boldf(\bu)]^*\bp$ as the adjoint equation. We refer to both equations together as the adjoint system.

The adjoint system, together with the variational equations of \eqref{ODEmanifold}, possesses a quadratic invariant. The variational equation is given by considering the tangent lifted vector field on $TU$, $\widetilde{\boldf}: TU \rightarrow TTU$, which is defined in terms of the flow $\Phi_{\epsilon}$ generated by $\boldf$ by
$$ \widetilde{\boldf}(\bu,\delta \bu) = \frac{d}{d\epsilon}\Big|_{\epsilon = 0} T\Phi_{\epsilon} (\bu,\delta \bu), $$
where $(\bu,\delta \bu)$ are coordinates on $TU$ and we refer to $\delta \bu$ as the variation. That is, $\widetilde{\boldf}$ is the infinitesimal generator of the tangent lifted flow. The variational equation associated with \eqref{ODEmanifold} is the ODE associated with the tangent lifted vector field. In coordinates, 
\begin{equation}\label{VariationalODEmanifold}
\frac{d}{dt}(\bu,\delta \bu) = \widetilde{\boldf}(\bu,\delta \bu).
\end{equation}
We interpret the state variable, the adjoint variable, and the variation together, $(\bu(t), \bp(t), \delta \bu(t))$, as a curve on the Whitney sum of the cotangent and tangent bundles of $U$ fibered over $U$, denoted $T^*U \oplus_U TU$, whose fiber over $\bu \in U$ is $T^*_\bu U \times T_\bu U$. Such a curve enjoys the following adjoint-variational conservation law.

\begin{prop}[Proposition 2.3 of \cite{TrLe2024}] \label{prop:ManifoldQuadraticInvariant}

For a curve $(\bu,\delta \bu,\bp)$ on $TU \oplus_U T^*U$ satisfying \eqref{AdjointODEmanifold} and  \eqref{VariationalODEmanifold},
\begin{equation}\label{ManifoldVariationalQuadratic}
\frac{d}{dt} \llangle \bp(t), \delta \bu(t)\rrangle_{\bu(t)} = 0. \end{equation} \qed
\end{prop}
This quadratic invariant is the key to adjoint sensitivity analysis \cite{Sa2016}, and its invariance is a consequence of the symplecticity of the Hamiltonian flow \cite{TrLe2024}. 

\section{Preconditioning of adjoint systems for evolution equations}\label{sec:adjoint-preconditioning-general}
We now discuss preconditioning transformations of adjoint systems. To begin, we will motivate these transformations by discussing preconditioning of standard gradient descent.

\subsection{Derivatives, gradients and duality}\label{sec:duality}
Let $U$ be a finite-dimensional vector space equipped with an inner product $((\cdot,\cdot)): U \times U \rightarrow \mathbb{R}$ with induced norm $\|\cdot\|_U$. Let $Z$ be a normed finite-dimensional vector space with norm $\|\cdot\|_Z$. 

\begin{definition}[Fr\'{e}chet Derivative] \label{def:frechet-derivative}
We say that a mapping $h: U \rightarrow Z$ is Fr\'{e}chet differentiable if there exists a mapping $Dh: U \rightarrow L(U,Z)$, where $L(U,Z)$ denotes the space of linear maps from $U$ to $Z$, such that, for all $\bu \in U$, 
\begin{equation}\label{eq:frechet-derivative}
    \lim_{\|\delta \bu \|_U\rightarrow 0} \frac{\|h(\bu + \delta \bu) - h(\bu) - Dh(\bu) \delta \bu \|_Z}{\|\delta \bu\|_U} = 0.
\end{equation}
If $h$ is Fr\'{e}chet differentiable, we refer to $Dh$ as the \textbf{Fr\'{e}chet derivative} (or simply, the derivative) of $h$.
\end{definition}

We will assume throughout that all maps are Fr\'{e}chet differentiable. We will be particularly interested in the cases when the codomain $Z = \mathbb{R}$, i.e., the mapping is a scalar function, and when the codomain $Z = U$, i.e., the mapping is a vector field.

Let us first consider the case when the codomain is $\mathbb{R}$. Let $C: U \rightarrow \mathbb{R}$; by \Cref{def:frechet-derivative}, its Fr\'{e}chet derivative is a mapping $DC: U \rightarrow L(U, \mathbb{R}) =: U^*$, where $U^* = L(U, \mathbb{R})$ is the dual space consisting of real-valued linear functionals on $U$. To contextualize this discussion, consider the gradient descent algorithm for the unconstrained optimization problem $\min_{\bu \in U} C(\bu)$,
\begin{equation}\label{eq:grad-descent}
    \bu^{k+1} = \bu^k - \gamma \nabla C(\bu^k),
\end{equation}
where $\gamma$ is the step-size and $\bu^0$ is some initial iterate. Here, the gradient $\nabla C: U \rightarrow U$ is defined in terms of the Fr\'{e}chet derivative $DC: U \rightarrow U^*$, the inner product on $U$, and a duality pairing, which we will now define.  

\begin{definition}[Duality Pairing]\label{def:duality-pairing}
A \textbf{duality pairing} between $U$ and its dual $U^*$ is a non-degenerate bilinear mapping $\langle\cdot,\cdot\rangle: U^* \times U \rightarrow \mathbb{R}$. 

Viewing $U^*$ as the space of linear functionals $L(U, \mathbb{R})$, the \textbf{standard duality pairing} $\llangle\cdot,\cdot\rrangle$ is defined by
$$ \llangle \bp, \bu\rrangle := \bp(\bu) $$
for all $\bp \in L(U,\mathbb{R}) = U^*$ and all $\bu \in U$.
\end{definition}

Now, for any $\bp \in U^*$, by the Riesz representation theorem, there exists a unique vector $\bp^\sharp \in U$ such that
\begin{equation}\label{eq:RRT}
    \langle \bp,\bu\rangle = ((\bp^\sharp, \bu)) \text{ for all } \bu \in U.
\end{equation}
Note that $\bp^\sharp$ depends on both the choice of duality pairing $\langle\cdot,\cdot\rangle$ and the choice of inner product $((\cdot,\cdot))$. For given choices of a duality pairing and inner product, the \textbf{gradient} of $C$ is a mapping $\nabla C: U \rightarrow U$ defined by
\begin{equation}\label{eq:gradient}
    \nabla C(\bu) := (DC(\bu))^\sharp \in U.
\end{equation}

 Observe that the gradient descent algorithm can be interpreted as follows: one chooses an inner product and duality pairing to identify the derivative $DC(\bu) \in U^*$ with a vector $\nabla C(\bu) \in U$, and one then performs a greedy minimization of the objective with respect to $\nabla C(\bu)$. An alternative to the gradient descent algorithm \eqref{eq:grad-descent} is the mirror descent algorithm \cite{mirrordescent1, mirrordescent2, mirrordescent3}. In mirror descent, one transforms the state vector $\bu \in U$ into a dual vector $\widetilde{\bu}$ in $U^*$, performs the descent step in the dual space with the derivative $DC(\bu) \in U^*$, and then transforms the resulting dual vector back into a state vector. More precisely, consider a generally nonlinear invertible transformation $F: U \rightarrow U^*$. The mirror descent algorithm is given by
\begin{subequations}\label{eq:mirror-descent}
    \begin{align}
        \widetilde{\bu}_k &:= F(\bu_k) \in U^*, \\
        \widetilde{\bu}_{k+1} &:= \widetilde{\bu}_k - \gamma DC(\bu_k) \in U^*, \\
        \bu_{k+1} &:= F^{-1} (\widetilde{\bu}_{k+1}) \in U.
    \end{align}
\end{subequations}

\textbf{Duality on $\mathbb{R}^N$.}
In practice, one typically works with vectors in $\mathbb{R}^N \ni \vec{u}$ instead of abstract vector spaces $U$, $U^*$ (throughout, we will denote vectors in $\mathbb{R}^N$ with an overset vector symbol). One can interpret working with $\mathbb{R}^N$ as coordinatizing a vector space relative to a given basis. Consider a finite-dimensional vector space $U$ with a basis $\{\varphi_i \in U\}_{i=1}^N$. Then, $U$ is coordinatized by the coefficients $\vec{u} = (u^i) \in \mathbb{R}^N$ in a basis expansion, i.e.,
    $$ \vec{u} = (u^i) \in \mathbb{R}^N \text{ coordinatizes } \bu = \sum_{i=1}^{N} u^i \varphi_i \in U.$$
Similarly, the dual space $U^*$ can be coordinatized by the coefficients $\vec{p} = (p_j) \in \mathbb{R}^N$ relative to a basis $\{l^j \in L(U,\mathbb{R})\}_{j=1}^N$ of $U^*$, i.e.,
    $$ \vec{p} = (p_j) \in \mathbb{R}^N \text{ coordinatizes } \bp = \sum_{j=1}^N p_j l^j \in U^*.$$
Although $\vec{u}$ and $\vec{p}$ above are both elements of $\mathbb{R}^N$, we will refer to $\vec{u}$ as a \emph{primal} vector, in the sense that it coordinatizes an element $\bu \in U$, and we will refer to $\vec{p}$ as a \emph{dual} vector, in the sense that it coordinatizes an element $\bp \in U^*$.

In this setting, the standard inner product on $\mathbb{R}^N$ is given by
$$ ((\vec{u}_1,\vec{u}_2))_{\text{standard}} = \vec{u}_1^{\,T} \vec{u}_2. $$
Similarly, for a dual vector $\vec{p} \in \mathbb{R}^N$ and primal vector $\vec{u} \in \mathbb{R}^N$, the standard duality pairing is given by
$$ \llangle \vec{p}, \vec{u} \rrangle = \vec{p}^{\,T} \vec{u}. $$

Identifying the space of linear transformations $L(\mathbb{R}^N, \mathbb{R}^M)$ with the space $\mathbb{R}^{M \times N}$ of real $M \times N$ matrices, a choice of inner product $((\cdot,\cdot))_M$ on $\mathbb{R}^N$ is equivalent to a choice of symmetric positive-definite (SPD) matrix $M \in \mathbb{R}^{N \times N}$, defined by
$$ ((\vec{u}_1,\vec{u}_2))_M := \vec{u}_1^{\,T} M \vec{u}_2. $$
Similarly, a choice of duality pairing $\langle\cdot,\cdot\rangle$ is equivalent to a choice of invertible matrix $P \in \mathbb{R}^{N \times N}$, defined by
$$ \langle \vec{p},\vec{u}\rangle = \vec{p}^{\,T} P \vec{u}. $$
Thus in $\mathbb{R}^N$, the identification of a dual vector $\vec{p}$ with a primal vector $\vec{p}^{\,\sharp}$ from \eqref{eq:RRT}, with the above choice of inner product and duality pairing, is given by
$$ \vec{p}^{\,\sharp} = M^{-1} P \vec{p}. $$
The gradient descent algorithm in the $\mathbb{R}^N$ setting is then
\begin{equation}\label{eq:grad-descent-Rn}
    \vec{u}_{k+1} = \vec{u}_k - \gamma M^{-1} P D C(\vec{u}_k),
\end{equation}
where the cost function is a map $C: \mathbb{R}^N \rightarrow \mathbb{R}$ and the components of $DC$ are simply its partial derivatives with respect to the components of $\vec{u}$, i.e.,
$$ (DC(\vec{u}) )_i = \frac{\partial C(\vec{u})}{\partial u^i}. $$

\textbf{Preconditioning of gradient descent.} From \eqref{eq:grad-descent-Rn}, the freedom in choosing both a duality pairing and an inner product allows one to ``precondition" gradient descent, e.g.  \cite{precGD1, precGD2, precGD3}, by reshaping the derivative of the cost function by an SPD matrix $M^{-1}$ (corresponding to a choice of inner product) or by an invertible matrix $P$ (corresponding to a choice of duality pairing). Since $P$ is an arbitrary invertible matrix, one can absorb the SPD matrix $M^{-1}$ appearing in \eqref{eq:grad-descent-Rn} into the definition of $P$ by redefining $M^{-1}P \rightarrow P$. Geometrically, this amounts to choosing a fixed inner product and allowing the duality pairing to vary, which offers the same flexibility as allowing both the inner product and duality pairing to vary. Note that, in general, this is not the same as fixing a duality pairing and allowing the inner product to vary, since then one restricts the preconditioning of gradient descent to only SPD matrices. Thus, the preconditioned gradient descent becomes
\begin{equation}\label{eq:prec-grad-descent-Rn-1}
    \vec{u}_{k+1} = \vec{u}_k - \gamma P D C(\vec{u}_k). 
\end{equation}
There is a further degree of flexibility one can provide in the preconditioned gradient descent \eqref{eq:prec-grad-descent-Rn-1}, which is to allow the invertible matrix $P \in \mathbb{R}^{N \times N}$ to be a state-dependent matrix, i.e., $P(\bu) \in \mathbb{R}^{N \times N}$. Such a choice arises, for example, in the natural gradient method \cite{natgrad0, natgrad1} where $P(\bu)$ corresponds to (the inverse of) a Riemannian metric on $\mathbb{R}^N$. Thus, we have the state-dependent preconditioned gradient descent
\begin{equation}\label{eq:prec-grad-descent-Rn-2}
    \vec{u}_{k+1} = \vec{u}_k - \gamma P(\vec{u}_k) D C(\vec{u}_k). 
\end{equation}
As another degree of flexibility, inspired by the mirror descent algorithm \eqref{eq:mirror-descent}, one can transform the primal space prior to performing the gradient descent step. To summarize, we have discussed three possible transformations of the gradient descent algorithm:
\begin{itemize}
    \item Reshape the derivative by an invertible matrix as in \eqref{eq:prec-grad-descent-Rn-1}.
    \item Reshape the derivative by a state-dependent invertible matrix as in \eqref{eq:prec-grad-descent-Rn-2}.
    \item Transform the primal space as in \eqref{eq:mirror-descent}.
\end{itemize}
The above discussion is in the context of gradient-based methods for unconstrained optimization, but it inspires the transformations of the adjoint system that we now discuss. In \Cref{sec:adjoint-prec-main}, we will introduce analogous transformations for the adjoint system in the context of dynamically constrained optimization. Note that the second item above, of course, includes the first item as a special case; however, in the context of adjoint systems, we will cover the two cases separately since the state-dependent case requires more care to handle. We will refer to such transformations as \emph{adjoint preconditioning} since the transformations are conceptually similar to preconditioned gradient descent methods in unconstrained optimization. We will further see that these adjoint preconditioning transformations are natural in the sense that they appropriately pull back the symplectic (particularly, Hamiltonian) structure of the canonical adjoint system.

\subsection{Adjoint preconditioning}\label{sec:adjoint-prec-main}
In analogy with the above preconditioning transformations of gradient descent, we will now introduce several different notions of preconditioning of the adjoint system. First, we will discuss how a transformation of the state dynamics induces preconditioning of the adjoint equation. Subsequently, we will discuss how the adjoint equation can be preconditioned independently from the state dynamics, which corresponds to a change in duality pairing. We will then generalize this further by allowing the duality pairing to be state-dependent.

To begin, we will review the canonical adjoint system in the context of dynamically constrained optimization. Consider the following optimization problem subject to the state dynamics of an ODE,
    \begin{align}\label{eq:state-constrained-optimization}
        \min_{\bu_0 \in U} C( &\bu(T) )  \text{ such that } \\
         \dot{\bu} &= \boldf(t,\bu) \text{ for } t \in (0,T), \nonumber \\
        \bu(0) &= \bu_0, \nonumber
    \end{align}
where $\boldf: [0,T] \times U \rightarrow U$ is a generally nonlinear time-dependent vector field on $U$. As is well-known \cite{CaLiPeSe2003, Sa2016, EiLa2020, LiPe2004, MaMiYa2023, NoWa2007, SaKa2000}, such dynamically-constrained optimization problems can be approached using gradient-based methods through the adjoint equation, where one regards $C(\bu(T))$ implicitly as a function of $\bu_0$ through the evolution equation $\dot{\bu} = \boldf(t,\bu), \bu(0) = \bu_0$. Note that it is straightforward to modify the adjoint-based approach to instead consider optimizing cost functions over parameters, as opposed to the initial condition, as well as including a running cost term in the cost function (see, e.g., \cite{TrLe2024, TrLe2025}).

Roughly speaking, the adjoint equation converts the derivative $DC(\bu(T))$ at the final time into its derivative with respect to the initial condition $\bu_0$, i.e., ``backpropagates" the derivative \cite{Sa2016}. This is illustrated in \Cref{fig:backprop_0}.
\begin{figure}[H]
\begin{center}
\includegraphics[width=80mm]{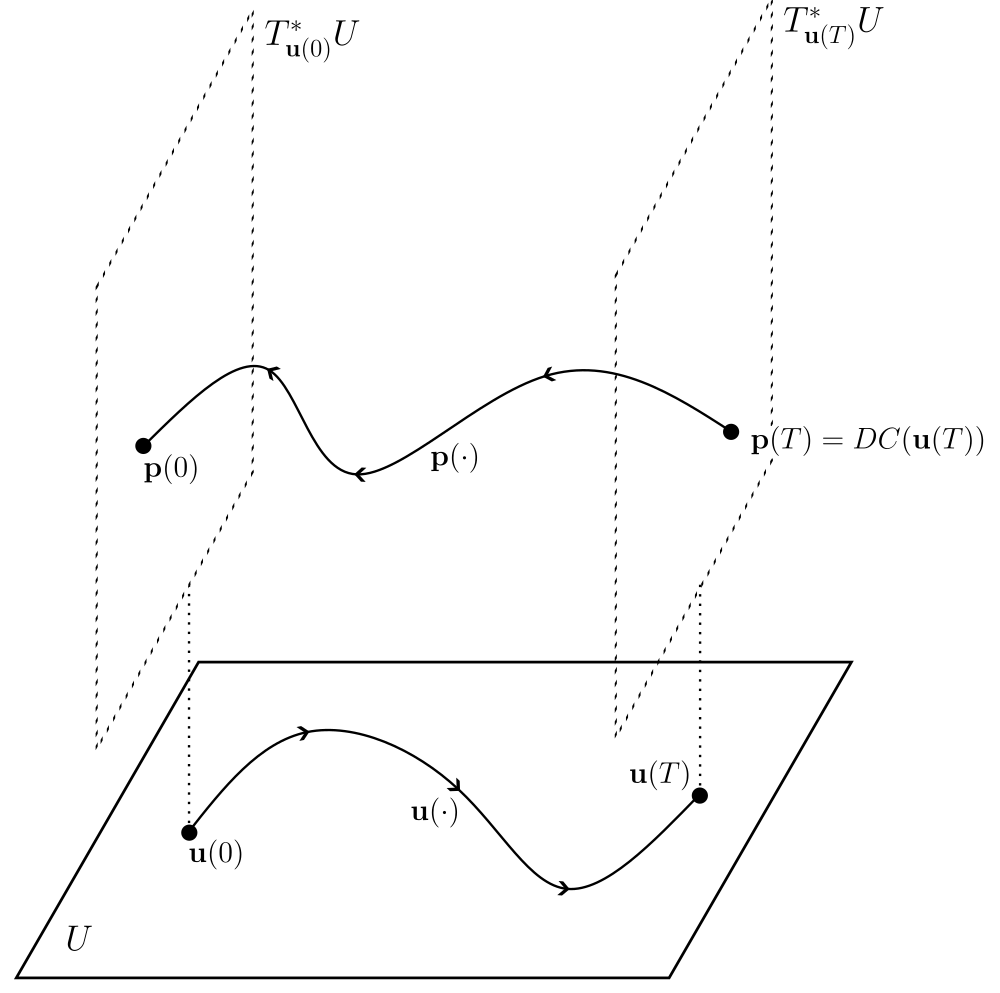}
\caption{Schematic of the state dynamics for $\bu(\cdot)$ evolving on the base space $U$ and backpropagation through the adjoint equation for $\bp(\cdot)$ evolving in reverse time on the fibers of $T^*U$ over $\bu(\cdot)$.}
\label{fig:backprop_0}
\end{center}
\end{figure}

In particular, the adjoint system on $U \times U^*$ associated with the problem \eqref{eq:state-constrained-optimization} is given by the state dynamics on $U$ with a specified initial condition, together with the adjoint equation on $U^*$ with a specified terminal condition, i.e.,
\begin{subequations}\label{eq:vector-space-adjoint}
\begin{align}
    \frac{d}{dt} \bu(t) &= \boldf(t,\bu) \text{ for } t \in (0,T), \label{eq:vector-space-adjoint-a}\\
    \frac{d}{dt} \bp(t) &= - [D\boldf(t,\bu(t))]^* \bp(t) \text{ for } t \in (0,T),\label{eq:vector-space-adjoint-b} \\
    \bu(0) &= \bu_0,\\
    \bp(T) &= DC(\bu(T)).
\end{align}
\end{subequations}
In \eqref{eq:vector-space-adjoint} above, for each $t$, $\boldf(t,\cdot)$ is a map $U \rightarrow U$; thus, its derivative is a map $D\boldf(t,\bu(t)) \in L(U,U)$. Furthermore, $[\ \cdot\ ]^*$ denotes the adjoint of this mapping with respect to the standard duality pairing, $[D\boldf(t,\bu(t))]^* \in L(U^*,U^*)$, defined by
$$\llangle\, [D\boldf(t,\bu(t))]^* \bq, \bv\rrangle = \llangle \bq, D\boldf(t,\bu(t)) \bv \rrangle \text{ for all } \bq \in U^*, \bv \in U. $$
Then, when $\bp(t)$ is a solution of the adjoint equation \eqref{eq:vector-space-adjoint}, $\bp(0)$ gives the derivative of $C(\bu(T))$ with respect to $\bu_0$ \cite{Sa2016, TrLe2024}, where $\bu(T)$ is given by the flow of $\boldf$ on $\bu_0$ from time $0$ to time $T$. A gradient descent algorithm for the problem \eqref{eq:state-constrained-optimization} then takes the form
\begin{equation}\label{eq:adjoint-grad-descent-algorithm}
    \bu_0^{k+1} = \bu^{k}_0 - \gamma (\bp(0)^k)^\sharp,
\end{equation}
where $(\cdot)^\sharp$, defined in \Cref{eq:RRT}, is taken with respect to a given inner product on $U$ and the standard duality pairing on $U^*$, $\bu^0_0 = \bu_0$ is the initial iterate, and $\bp(0)^k$ is the solution of the adjoint equation \eqref{eq:vector-space-adjoint} with terminal condition $\bp(T) = DC(\bu^k(T))$, where $\bu^k(\cdot)$ satisfies the state dynamics $\dot{\bu}^k = \boldf(t,\bu^k)$ with initial condition $\bu^k(0) = \bu^k_0$.

As discussed in \Cref{sec:geometry-intro}, the adjoint system corresponds to a Hamiltonian system with the canonical symplectic structure on $T^*U$. The fact that the adjoint equation backpropagates the derivative arises from the adjoint-variational conservation law \cite{Sa2016},
$$ \frac{d}{dt} \llangle \bp, \delta\bq \rrangle = 0, $$
which in turn follows from the symplecticity of the Hamiltonian flow of the adjoint system \cite{TrLe2024}. We will now discuss three different types of adjoint preconditioning and furthermore, show that they are natural as they appropriately transform the symplectic structure of the canonical adjoint system and, thus, admit analogous adjoint-variational conservation laws.

\subsubsection{Adjoint preconditioning induced from transformation of the state dynamics.}\label{sec:induced-from-state}
In this section, we discuss how a transformation of the state dynamics induces an adjoint preconditioning transformation. We will then provide an example where such a transformation arises: evolution equations involving mass matrices. To conclude this section, we will interpret this adjoint preconditioning transformation as a symplectomorphism preserving the canonical Hamiltonian structure of the adjoint system.

Consider a transformation of the base space, given by a generally nonlinear diffeomorphism
\begin{align*}\label{eq:base-space-transformation}
    \bL: U &\rightarrow U, \\
        \bu &\mapsto \widetilde{\bu} := \bL(\bu).
\end{align*}
Under this transformation, the state dynamics \eqref{eq:vector-space-adjoint-a} transform as
\begin{equation}\label{eq:U-state-dynamics-transformed}
    \frac{d}{dt} \widetilde{\bu} = D\bL( \bL^{-1} (\widetilde{\bu}))\boldf(t,\bL^{-1}(\widetilde{\bu})). 
\end{equation}
We now compute the adjoint equation for the transformed state dynamics \eqref{eq:U-state-dynamics-transformed}. Denoting the right-hand side of \eqref{eq:U-state-dynamics-transformed} as $\widetilde{\boldf}(t,\widetilde{\bu})$, the corresponding adjoint equation is $d\widetilde{\bp}/dt = - [D\widetilde{\boldf}(t,\widetilde{\bu})]^* \widetilde{\bp}$. Expanding out the derivative on the right-hand side yields
\begin{align}\label{eq:adjoint-state-dynamics-transformed}
    \frac{d}{dt}\widetilde{\bp} &= - (D\bL^{-1}(\widetilde{\bu}))^* [D\boldf(t, \bL^{-1}(\widetilde{\bu}))]^* D\bL(\bL^{-1}(\widetilde{\bu}))^* \widetilde{\bp} \\
    & \qquad - (D\bL^{-1}(\widetilde{\bu}))^* \Big(D^2\bL(\bL^{-1}(\widetilde{\bu})) \boldf(t,\bL^{-1}(\widetilde{\bu}) )\Big) ^* \widetilde{\bp}, \nonumber
\end{align}
where the second derivative, $D^2\bL$, is defined as follows. Since $\bL$ is a map $U \rightarrow U$, its first derivative is a map $D\bL: U \rightarrow L(U,U)$; thus, its second derivative is a map
$$ D^2\bL: U \rightarrow L(U, L(U,U)).$$
Evaluating this map at $\bL^{-1}(\widetilde{\bu})$ yields a linear map
$$ D^2\bL(\bL^{-1}(\widetilde{\bu})) : U \rightarrow L(U,U). $$
Subsequently, evaluating this map at $\boldf(t,\bL^{-1}(\widetilde{\bu}) )$ yields an operator in $L(U,U)$, i.e.,
$$ D^2\bL(\bL^{-1}(\widetilde{\bu})) \boldf(t,\bL^{-1}(\widetilde{\bu}) ) \in L(U,U). $$
The adjoint of this operator is the quantity that appears in \eqref{eq:adjoint-state-dynamics-transformed}. We refer to equations \eqref{eq:U-state-dynamics-transformed}-\eqref{eq:adjoint-state-dynamics-transformed} as the preconditioned adjoint system induced from a transformation of the state dynamics.

A special case of this preconditioned adjoint system is when $\bL$ is linear, $\widetilde{\bu} = \bL(\bu) = A\bu$, where $A \in L(U,U)$ is invertible. In this case, since the second derivative of $\bL$ vanishes, the system simplifies to
\begin{subequations}\label{eq:adjoint-system-linear-transformation}
    \begin{align}
        \frac{d}{dt} \widetilde{\bu} &= A\boldf(t,A^{-1}\widetilde{\bu}), \label{eq:adjoint-system-linear-transformation-u} \\
        \frac{d}{dt}\widetilde{\bp} &= - A^{-*} [D\boldf(t, A^{-1}\widetilde{\bu})]^* A^* \widetilde{\bp},  \label{eq:adjoint-system-linear-transformation-p}
    \end{align}
\end{subequations}
where $A^{-*}$ denotes the adjoint of the inverse (equivalently, the inverse of the adjoint). We can rewrite \eqref{eq:adjoint-system-linear-transformation} so that it does not involve the inverse of $A$ or $A^*$; we do this by expressing the system in terms of $\bu$ and applying $A^*$ to both sides of \eqref{eq:adjoint-system-linear-transformation-p}, that is, 
\begin{subequations}\label{eq:adjoint-linear-original-variable}
\begin{align}
    \frac{d}{dt} A\bu &= A \boldf(t,\bu), \label{eq:adjoint-linear-original-variable-a}\\
    \frac{d}{dt} A^*\widetilde{\bp} &= - [D\boldf(t, \bu)]^* A^* \widetilde{\bp}.
\end{align}
\end{subequations}
The system \eqref{eq:adjoint-linear-original-variable} is often preferable to work with in practice, since it does not involve directly inverting $A$. We consider a commonly encountered example below.

\begin{example}[Evolution equations with mass matrices]\label{ex:mass-matrices}
\sloppy An obvious example of where the transformed equation \eqref{eq:adjoint-linear-original-variable-a} is preferred over the original state dynamics \eqref{eq:vector-space-adjoint-a} is upon semi-discretizing a PDE (e.g., using a finite element semi-discretization), where the finite-dimensional evolution equation involves an invertible mass matrix $M$. In such a case, in \eqref{eq:adjoint-linear-original-variable-a} with $A = M$, the vector field $\boldf$ is not formed explicitly; rather, the vector field $\boldF(t,\bu) := M\boldf(t,\bu)$ is. Expressing the preconditioned adjoint system \eqref{eq:adjoint-linear-original-variable} in terms of $\boldF$ gives the expression
\begin{align*}
    \frac{d}{dt} M\bu &= \boldF(t,\bu), \\
    \frac{d}{dt} M^*\widetilde{\bp} &= - [D\boldF(t, \bu)]^* \widetilde{\bp}.
\end{align*}
For a more detailed discussion of how semi-discretization, particularly mass matrices, is incorporated into adjoint systems, see \cite{TrSoLe2024}. However, note that in \cite{TrSoLe2024}, we considered only a fixed (state-independent) mass matrix; however, the discussion of this section applies to generally nonlinear transformations of the state dynamics and, in particular, applies to systems with state-dependent mass matrices, which arise in, for example, Lagrangian hydrodynamics \cite{DoKoRi2012}, elastic multi-body dynamics \cite{PaVo92}, and rigid body dynamics \cite{VoGr21}.
\end{example}

\textbf{Adjoint preconditioning induced from state dynamics as a symplectic automorphism.}
Now, we will interpret this form of adjoint preconditioning, induced from a transformation of the state dynamics, as a canonical symplectomorphism. Note that the transformation of the standard adjoint system \eqref{eq:vector-space-adjoint-a}-\eqref{eq:vector-space-adjoint-b} into the preconditioned adjoint system \eqref{eq:U-state-dynamics-transformed}-\eqref{eq:adjoint-state-dynamics-transformed} arises from the following transformation of $T^*U = U \times U^*$, 
$$ (\bu, \bp) \mapsto (\widetilde{\bu}, \widetilde{\bp}) = (\bL(\bu), (D\bL^{-1}(\widetilde{\bu}))^*\bp). $$
This mapping is precisely the cotangent lift $T^*\bL: T^*U \rightarrow T^*U$ of the map $\bL: U \rightarrow U$ (for a discussion of the geometry of the cotangent bundle and lifts, see \cite{YaIs1973, MaRa1999}).

Denote the canonical symplectic form on $T^*U$ by
\begin{equation}\label{eq:canonical-symplectic-form}
     \Omega = - \llangle d\bp \wedge d\bu\rrangle,
\end{equation}
whose action on $(\bu_1,\bp_1), (\bu_2, \bp_2) \in U \times U^*$ is given by
$$ \Omega \cdot [(\bu_1,\bp_1), (\bu_2, \bp_2)] = -\llangle \bp_1, \bu_2\rrangle + \llangle \bp_2,\bu_1\rrangle. $$
Then, since $T^*\bL$ is a cotangent lift, it immediately follows that $\Omega$ is preserved by this map \cite{MaRa1999}, $(T^*\bL)^{\wedge}\Omega = \Omega$ (here, $(\ \cdot\ )^{\wedge}$ denotes the pullback of a map; we do not use the standard notation $(\ \cdot\ )^*$ since this was used throughout to denote the adjoint of an operator). In other words, $T^*\bL: (T^*U, \Omega) \rightarrow (T^*U, \Omega)$ is a symplectic automorphism (i.e., it is a symplectic diffeomorphism of $T^*U$ to itself, equipped with the same symplectic form in both the domain and codomain).

\sloppy Furthermore, recall that the standard adjoint system \eqref{eq:vector-space-adjoint-a}-\eqref{eq:vector-space-adjoint-b} is a (time-dependent) Hamiltonian system with respect to the canonical symplectic form and Hamiltonian $H(t,\bu,\bp) = \llangle\bp,\boldf(t,\bu)\rrangle$. Similarly, the preconditioned adjoint system \eqref{eq:U-state-dynamics-transformed}-\eqref{eq:adjoint-state-dynamics-transformed} is a Hamiltonian system with respect to the canonical symplectic form and Hamiltonian $H_\bL$, given by pulling back $H$ through the inverse of $T^*\bL$ (where we trivially extend $T^*\bL$ to a time-dependent map by $T^*\bL: [0,T] \times (T^*U, \Omega) \rightarrow [0,T] \times (T^*U, \Omega)$, $T^*\bL: (t,\bu,\bp) \mapsto (t, \bL(\bu), (D\bL^{-1}(\widetilde{\bu})\bp)$). That is,
$$ H_\bL(t,\widetilde{\bu},\widetilde{\bp}) := H(t,\bL^{-1}(\widetilde{\bu}),(D\bL(\bL^{-1}(\widetilde{\bu})))^*\widetilde{\bp}) = \llangle \widetilde{\bp}, D\bL( \bL^{-1} (\widetilde{\bu}))\boldf(t,\bL^{-1}(\widetilde{\bu})) \rrangle . $$
It follows that $T^*\bL$ is an extended symplectic automorphism, i.e.,
$$ \Omega + dH \wedge dt = (T^*\bL)^{\wedge} (\Omega + dH_\bL \wedge dt),$$
mapping the standard adjoint system \eqref{eq:vector-space-adjoint-a}-\eqref{eq:vector-space-adjoint-b} to the preconditioned adjoint system \eqref{eq:U-state-dynamics-transformed}-\eqref{eq:adjoint-state-dynamics-transformed}. The transformation is depicted in \Cref{fig:prec_state}.

\begin{figure}[H]
\begin{center}
\includegraphics[width=150mm]{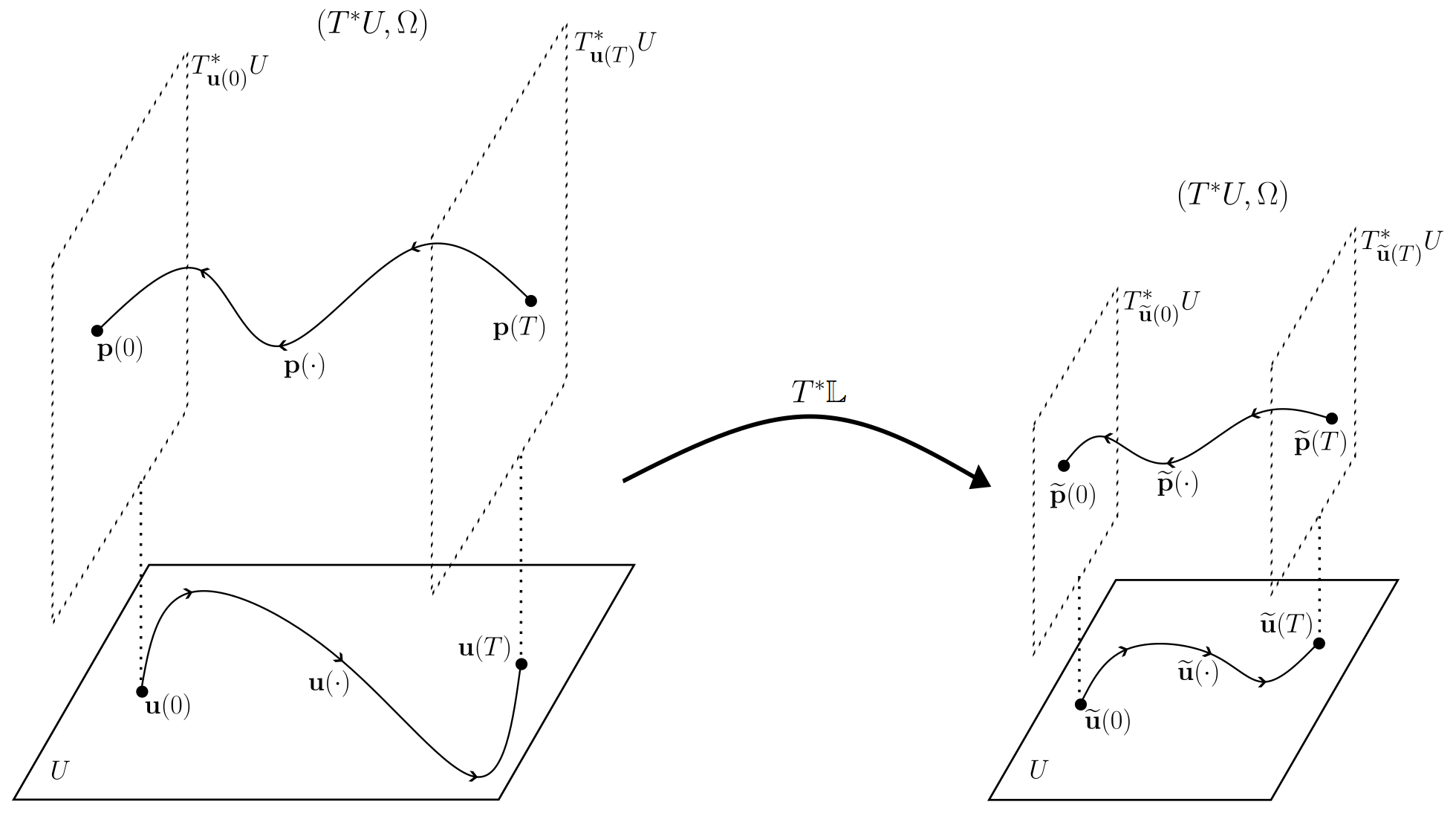}
\caption{Schematic of the mapping arising from the preconditioning of the adjoint system induced by a transformation of the state dynamics. The base space and the fibers are transformed under this map.}
\label{fig:prec_state}
\end{center}
\end{figure}

Finally, since $T^*\bL$ is an extended symplectic automorphism, the following proposition immediately follows. Let $\delta\widetilde{\bu}$ be the solution to the variational equation, given by linearizing the transformed state dynamics \eqref{eq:U-state-dynamics-transformed}, i.e.,
\begin{align}\label{eq:variational-eqn-nonlinear-transform}
        \frac{d}{dt} \delta \widetilde{\bu} &= D\bL( \bL^{-1} (\widetilde{\bu}))D\boldf(t,\bL^{-1}(\widetilde{\bu})) D\bL^{-1}(\widetilde{\bu})\delta\widetilde{\bu}  \\
        & \qquad + \Big(D^2\bL(\bL^{-1}(\widetilde{\bu})) \boldf(t,\bL^{-1}(\widetilde{\bu}) )\Big) D\bL^{-1}(\widetilde{\bu}) \delta \widetilde{\bu}. \nonumber
\end{align}
Then, we have the following adjoint-variational conservation law.
\begin{prop}\label{prop:prec-1}
    Let $\widetilde{\bu}$ satisfy the transformed state dynamics \eqref{eq:U-state-dynamics-transformed}, $\delta\widetilde{\bu}$ satisfy the corresponding variational equation \eqref{eq:variational-eqn-nonlinear-transform}, and $\widetilde{\bp}$ satisfy the corresponding preconditioned adjoint equation \eqref{eq:adjoint-state-dynamics-transformed}. Then,
    \begin{equation}
        \frac{d}{dt} \llangle \widetilde{\bp}(t), \delta\widetilde{\bu}(t)\rrangle = 0.
    \end{equation} \qed
\end{prop}

\subsubsection{Adjoint preconditioning induced by a duality pairing.}\label{sec:induced-from-duality}
Compared to the previous form of adjoint preconditioning, which was induced from a transformation of the state dynamics, we would also like to consider preconditioning of the adjoint equation independent of the state dynamics. Our motivation for this consideration is that, for adjoint-based optimization, one may wish to reshape the derivative obtained through the adjoint backpropagation without changing the dynamics on the base space. We will subsequently interpret this form of adjoint preconditioning as a symplectic isomorphism between two different symplectic spaces, $(T^*U, \Omega) \rightarrow (T^*U, \Omega_P)$ (here, the underlying space is the same, but the symplectic forms are different).

Consider the state dynamics on $U$ \eqref{eq:vector-space-adjoint-a} with its associated adjoint equation \eqref{eq:vector-space-adjoint-b}. Note that the adjoint equation \eqref{eq:vector-space-adjoint-b} on $U^*$ is formulated with respect to the standard duality pairing $\llangle\cdot,\cdot\rrangle$. However, one could instead formulate the adjoint equation with respect to a different duality pairing $\langle\cdot,\cdot\rangle_P$. Now, if $\langle\cdot,\cdot\rangle_P: U^* \times U \rightarrow \mathbb{R}$ is a duality pairing (see \Cref{def:duality-pairing}), there exists an automorphism (i.e., invertible linear transformation) $P$ of $U$ such that $\langle\cdot,\cdot\rangle_P$ is related to the standard duality pairing by \cite{TrSoLe2024}
$$ \langle\cdot,\cdot\rangle_P = \llangle\cdot,P\, \cdot\rrangle. $$
The adjoint equation with respect to the duality pairing $\langle\cdot,\cdot\rangle_P$ can be expressed \cite[Eq. 3.2b]{TrSoLe2024})
    \begin{align}\label{eq:two-scale-adjoint-DE-prec}
        \frac{d}{dt} \boldsymbol{\xi} &= - [D\boldf(t,\bu)]^{*P} \boldsymbol{\xi} \\
        &=- P^{-*}[D\boldf(t,\bu)]^* P^* \boldsymbol{\xi}, \nonumber
    \end{align}
where, here, $[\ \cdot\ ]^{*P}$ denotes the adjoint with respect to $\langle\cdot,\cdot\rangle_P$, and in the second equality, we have expressed this in terms of the adjoint with respect to the original pairing, $[\ \cdot\ ]^*$. We interpret \eqref{eq:two-scale-adjoint-DE-prec} as a preconditioning of the adjoint system with respect to a modified duality pairing. We will discuss an example of this type of preconditioning in the context of coupled evolution equations in \Cref{sec:scale-preconditoning-coupled}.

As we will now discuss, the transformation $\bp \mapsto \boldsymbol{\xi}$ induced from the duality pairing $\langle\cdot,\cdot\rangle_P$ corresponds to a symplectomorphism of the cotangent bundle, transforming the fibers while leaving the base space $U$ invariant.

\textbf{Adjoint preconditioning induced by duality pairing as a symplectic isomorphism.}
The transformation of the adjoint system \eqref{eq:vector-space-adjoint-a}-\eqref{eq:vector-space-adjoint-b} into the preconditioned adjoint system \eqref{eq:vector-space-adjoint-a}, \eqref{eq:two-scale-adjoint-DE-prec} is given by the map
\begin{subequations}\label{eq:P-bar}
    \begin{align}
    \overline{P}: T^*U &\rightarrow T^*U, \\
                (\bu, \bp) & \mapsto (\bu, \boldsymbol{\xi}) = (\bu, P^{-*}\bp).
    \end{align}
\end{subequations}
Note that $\overline{P}$ transforms only the fibers of $T^*U$, while leaving the base space $U$ invariant. We can interpret this as a symplectic isomorphism of two different symplectic spaces. Namely, let $\Omega$ denote the canonical symplectic form \eqref{eq:canonical-symplectic-form} on $T^*U$ as before. Let $\Omega_P$ denote the symplectic form on $T^*U$ given by
$$ \Omega_P = - \langle d\boldsymbol{\xi} \wedge d\bu\rangle_P, $$
\sloppy where the notation $-\langle d\boldsymbol{\xi} \wedge d\bu\rangle_P$ for this two-form is defined by its action on vectors $(\bu_1,\boldsymbol{\xi}_1), (\bu_2, \boldsymbol{\xi}_2) \in U \times U^*$ via
$$ -\langle d\boldsymbol{\xi} \wedge d\bu\rangle_P \cdot [(\bu_1,\boldsymbol{\xi}_1), (\bu_2, \boldsymbol{\xi}_2)] = -\langle \boldsymbol{\xi}_1, \bu_2\rangle_P + \langle \boldsymbol{\xi}_2,\bu_1\rangle_P. $$
Alternatively, in coordinates $\bu = (u^\beta)$, $\boldsymbol{\xi} = (\xi_\alpha)$, and $P = (P^\alpha_\beta)$, the two-form $\langle d\boldsymbol{\xi} \wedge d\bu\rangle_P$ has the coordinate expression
$$ \langle d\boldsymbol{\xi} \wedge d\bu\rangle_P = P^\alpha_\beta d\xi_\alpha \wedge du^\beta. $$

From the above definitions, it immediately follows that $\overline{P}: (T^*U, \Omega) \rightarrow (T^*U, \Omega_P)$ is a symplectic isomorphism (between different symplectic spaces sharing the same underlying vector space), i.e.,
$$ \overline{P}^{\wedge}\Omega_P = \Omega. $$

As in the previous case, the associated time-dependent map is an extended symplectic isomorphism between the standard adjoint system and the preconditioned adjoint system, where the standard adjoint system is a Hamiltonian system with respect to the canonical symplectic form and Hamiltonian $H(t,\bu,\bp) = \llangle\bp,\boldf(t,\bu)\rrangle$, whereas the preconditioned adjoint system is a Hamiltonian system with respect to the symplectic form $\Omega_P$ and Hamiltonian $H_P$ related to $H$ by
$$ H_P(t,\bu,\boldsymbol{\xi}) = H(t,\bu,P^*\boldsymbol{\xi}) = \llangle \boldsymbol{\xi}, P \boldf(t,\bu)\rrangle = \langle \boldsymbol{\xi}, \boldf(t,\bu)\rangle_P . $$
The transformation is depicted in \Cref{fig:prec_duality}.

\begin{figure}[H]
\begin{center}
\includegraphics[width=150mm]{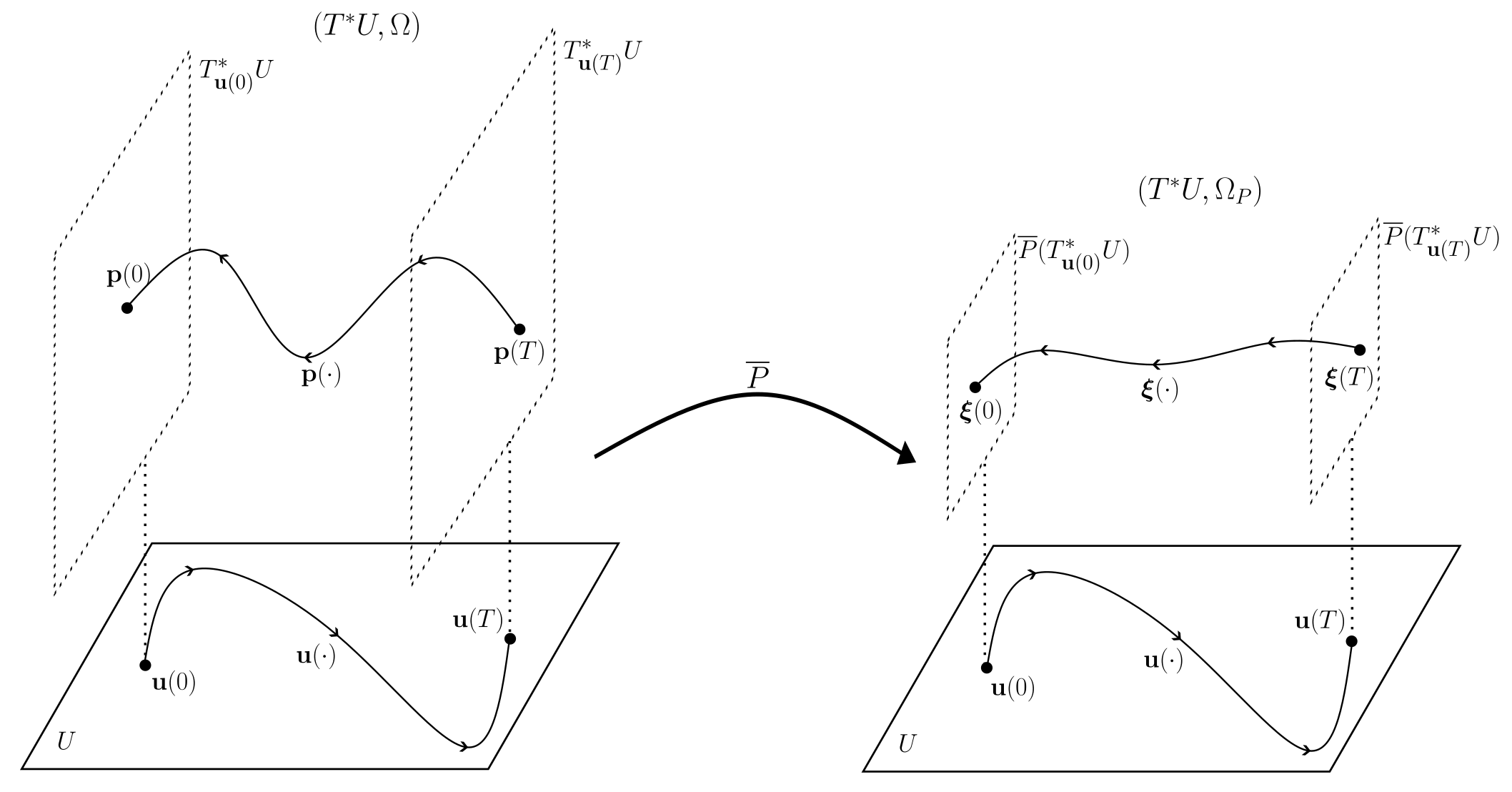}
\caption{Schematic of the mapping arising from the preconditioning of the adjoint system induced by a duality pairing. Only the fibers are transformed under this map, while the base space is left invariant.}
\label{fig:prec_duality}
\end{center}
\end{figure}

This yields an analogous conservation law to \Cref{prop:prec-1}.
\begin{prop}\label{prop:prec-2}
    Let $\bu$ satisfy the state dynamics \eqref{eq:vector-space-adjoint-a}, let $\delta\bu$ satisfy the associated variational equation, and let $\boldsymbol{\xi}$ satisfy the preconditioned adjoint equation with respect to the duality pairing $\langle\cdot,\cdot\rangle_P$, i.e., \eqref{eq:two-scale-adjoint-DE-prec}. Then,
    \begin{equation}
        \frac{d}{dt} \langle \boldsymbol{\xi}(t), \delta{\bu}(t)\rangle_P = 0.
    \end{equation} \qed
\end{prop}

\subsubsection{Adjoint preconditioning induced by a state-dependent duality pairing}\label{sec:fiberwise-duality}
In \Cref{sec:induced-from-duality}, we chose a non-standard duality pairing $\langle\cdot,\cdot\rangle_P$ in order to ``reshape" the adjoint variable $\bp \mapsto \boldsymbol{\xi} = P^{-*} \bp$ and its corresponding dynamics. This was a global transformation in the sense that $P$ did not depend on the base point. In order to provide an additional degree of flexibility, we can allow this transformation to be local, i.e., depend on the state $\bu$.
Thus, in this section, we will generalize \Cref{sec:induced-from-duality} to allow for a modified duality pairing which depends on the state $\bu$; roughly speaking, we will generalize the automorphism $P: U \rightarrow U$ of \Cref{sec:induced-from-duality} to be a state-dependent automorphism $P(\bu): U \rightarrow U$. As we will see, the corresponding adjoint system will have an induced Hamiltonian structure and interestingly, the adjoint equation can be naturally expressed in terms of a vector bundle connection on $T^*U$ defined in terms of the state-dependent automorphism. To formalize this, we introduce the following definitions.

Let $T^*U \oplus_U TU$ be the Whitney sum of the cotangent bundle of $U$ and the tangent bundle of $U$, fibered over $U$. We denote a typical element of this bundle by $(\bu, \bp,\delta \bu) \in T^*U \oplus_U TU$, i.e., $\bp \in T^*_\bu U$ and $\delta \bu \in T_\bu U$. Note that since $U$ is a vector space, this bundle has the trivialization $T^*U \oplus_U TU = U \times U^* \times U \ni (\bu,\bp,\delta \bu)$. We will interpret working with the bundle $TU \oplus_U T^*U$ as the \emph{fiberwise} picture, whereas working with the trivialization $U \times U \times U^*$ as the \emph{state-dependent} picture; both pictures are, of course, equivalent.

\begin{definition}[Fiberwise duality pairing]\label{def:fiber-duality}
    A \textbf{fiberwise (or state-dependent) duality pairing} is an assignment to each fiber of $T^*U \oplus_U TU$ a duality pairing; that is, for each $\bu \in U$, the fiber of this bundle is equipped with a non-degenerate bilinear form
    \begin{equation}\label{eq:fiber-duality-pairing}
        \langle \cdot,\cdot\rangle_{P(\bu)}: T_\bu^*U \times T_\bu U \rightarrow \mathbb{R}.
    \end{equation}
    Furthermore, we assume \eqref{eq:fiber-duality-pairing} varies smoothly with $\bu$.

    Equivalently, in the state-dependent picture, this is specified by a smooth map
    \begin{equation}\label{eq:fiber-duality-map}
        P: U \rightarrow L(U,U),
    \end{equation}
    such that $P(\bu)$ is an automorphism for each $\bu \in U$. The two pictures are related by
    $$ \langle \bp, \delta \bu\rangle_{P(\bu)} = \llangle \bp, P(\bu) \delta \bu\rrangle, $$
    where $\llangle\cdot,\cdot\rrangle$ is the standard duality pairing of $T^*U$ with $TU$.
\end{definition}

It is useful for intuition to provide coordinate expressions for the  following discussion. We coordinatize $T^*U \oplus_U TU$ with coordinates $(u^\alpha, p_\beta, \delta u^\gamma)$ such that the standard duality pairing is simply
$$ \llangle \bp, \delta\bu\rrangle = p_\alpha \delta u^\alpha; $$
throughout, we use the summation convention that repeated lower and upper indices are summed over. In coordinates, a fiberwise duality pairing is given by
$$ \langle \bp,\delta \bu\rangle_{P(\bu)} = p_\beta P^\beta_\alpha(\bu) \delta u^\alpha. $$
As we will see, the derivative of $P$ will arise in the discussion below. The Fr\'{e}chet derivative of $P: U \rightarrow L(U,U)$ is a mapping $DP: U \rightarrow L(U,L(U,U))$. Since $DP$ is a mapping valued in further nested spaces of mappings, we denote its evaluation on successive arguments as
\begin{align*}
    DP: U &\rightarrow L(U, L(U,U)), \\
        \bu &\mapsto DP(\bu): U \rightarrow L(U,U), \\
        & \hspace{5.7em} \bv \mapsto [DP(\bu)\bv ] (\ \cdot\ ): U \rightarrow U, \\
        & \hspace{14.4em} \delta\bu \mapsto [DP(\bu)\bv ] \delta\bu \in U.
\end{align*}
Note that both $[DP(\bu)\ \cdot\ ] \delta\bu$ and $[DP(\bu)\bv](\ \cdot\ )$ are elements of $L(U,U)$. In coordinates, the $\beta^{th}$ component of the derivative of $P$ is given by
$$ \Big([DP(\bu)\bv ] \delta\bu\Big)^\beta = D_\gamma P^\beta_\alpha(\bu)  v^\gamma \delta u^\alpha, $$
where $D_\gamma := \partial/\partial u^\gamma$. Now, we consider a state-dependent analogue of the transformation \eqref{eq:P-bar}; namely, consider the map
\begin{subequations}\label{eq:P-bar-fiberwise}
    \begin{align}
    \overline{P}: T^*U &\rightarrow T^*U, \\
                (\bu, \bp) & \mapsto (\bu, \boldsymbol{\xi}) := (\bu, P(\bu)^{-*}\bp).
    \end{align}
\end{subequations}
We will define the adjoint system by pulling back the canonical symplectic form and the adjoint Hamiltonian through the inverse of this map,
$$ \overline{P}^{-1}:  (\bu, \boldsymbol{\xi}) \mapsto (\bu, P(\bu)^*\boldsymbol{\xi}). $$
Let $\Omega$ be the canonical symplectic form on $T^*U$ as before. We define the following two-form on $T^*U$ as the pullback of $\Omega$ by $\overline{P}^{-1}$,
\begin{equation}\label{eq:P-fiber-form}
    \Omega_P := (\overline{P}^{-1})^{\wedge} \Omega.
\end{equation}
Unlike the previous case with a global transformation discussed in \Cref{sec:induced-from-duality}, it is not immediately obvious that $\Omega_P$ is indeed a symplectic form. Thus, we will prove this explicitly.
\begin{prop}
    The two-form $\Omega_P$ \eqref{eq:P-fiber-form} is a symplectic form, i.e., is closed and non-degenerate.
    \begin{proof}
        Clearly, $\Omega_P$ is closed as the pullback of a closed form. We will verify that $\Omega_P$ is non-degenerate in coordinates; this will be useful as it will provide a coordinate expression for $\Omega_P$, which we will subsequently use to compute Hamilton's equations.

        Using the expression \eqref{eq:canonical-symplectic-form} for the canonical symplectic form, its pullback by $\overline{P}^{-1}$ is given in coordinates by
        \begin{align}\label{eq:P-fiber-form-coordinates}
            \Omega_P &= (\overline{P}^{-1})^{\wedge} \Omega = - \llangle d(P(\bu)^*\boldsymbol{\xi}) \wedge d\bu \rrangle = -d(P^\beta_\alpha(\bu) \xi_\beta) \wedge du^\alpha \\
            &= -P^\beta_\alpha(\bu) d\xi_\beta \wedge du^\alpha - \xi_\beta D_\gamma P^\beta_\alpha(\bu) du^\gamma \wedge du^\alpha. \nonumber
        \end{align}
        Interestingly, in the last line of expression \eqref{eq:P-fiber-form-coordinates}, the first term is what we expect from a global change of duality pairing, as we saw in \Cref{sec:induced-from-duality}; however, there is now an additional term arising from the fact that the duality pairing varies with respect to the basepoint. As we will see, this will introduce an additional term in the adjoint equation. Interestingly, this additional term can be thought of as corresponding to a connection on the vector bundle $T^*U \rightarrow U$.
        
        Now, let $Z := z^\alpha \frac{\partial}{\partial u^\alpha} + \zeta_\beta \frac{\partial}{\partial \xi_\beta}$ be any vector satisfying $i_Z \Omega_P = 0$. We verify that $\Omega_P$ is non-degenerate, i.e., $i_Z \Omega_P = 0$ implies $Z = 0$. Compute
        \begin{align*}
           0 &= i_Z \Omega_P \\
           &= -P^\beta_\alpha(\bu) \zeta_\beta du^\alpha  + P^\beta_\alpha(\bu) z^\alpha d\xi_\beta - \xi_\beta D_\gamma P^\beta_\alpha(\bu) z^\gamma du^\alpha + \xi_\beta D_\gamma P^\beta_\alpha(\bu) z^\alpha du^\gamma.
        \end{align*}
        By linear independence of the coordinate differentials $du^\alpha$ and $d\xi_\beta$, the second term above implies $P^\beta_\alpha(\bu) z^\alpha = 0$ for all $\beta$. Since $P(\bu)$ is invertible, this implies $z^\alpha = 0$ for all $\alpha$. Substituting this into the remaining three terms, only the first term remains, which yields $P^\beta_\alpha(\bu) \zeta_\beta = 0$ for all $\alpha$. Again, since $P(\bu)^*$ is invertible, this implies $\zeta_\beta = 0$ for all $\beta$. Thus, $Z = 0$, i.e., $\Omega_P$ is non-degenerate. 
    \end{proof}
\end{prop}

Since $\Omega_P$ is a symplectic form on $T^*U$, we can thus define Hamilton's equation with respect to a given Hamiltonian. Let $H(t,\bu,\bp) = \llangle \bp, \boldf(t,\bu)\rrangle$ be the standard adjoint Hamiltonian. Then, we define the fiberwise-preconditioned adjoint Hamiltonian by pulling back $H$ through $\overline{P}^{-1}$ (trivially extended to a time-dependent map), i.e.,
\begin{align}
    H_P &:= (\overline{P}^{-1})^{\wedge} H, \\
    H_P(t,\bu,\boldsymbol{\xi}) &= H(t,\bu, P(\bu)^*\boldsymbol{\xi}) = \llangle\boldsymbol{\xi}, P(\bu) \boldf(t,\bu) \rrangle = \langle \boldsymbol{\xi}, \boldf(t,\bu) \rangle_{P(\bu)}. \nonumber
\end{align}
Unsurprisingly, the pulled back Hamiltonian is the adjoint Hamiltonian, with the standard duality pairing replaced by the fiberwise duality pairing. Now, we define the \emph{fiberwise-preconditioned} adjoint system as Hamilton's equation of motion with Hamiltonian $H_P$ and symplectic form $\Omega_P$. A straightforward calculation using the coordinate expression for the symplectic form \eqref{eq:P-fiber-form-coordinates} yields
\begin{subequations}\label{eq:fiberwise-prec-adjoint-system}
    \begin{align}
        \frac{d\bu}{dt} &= \boldf(t,\bu), \label{eq:fiberwise-prec-adjoint-system-a}\\
        \frac{d\boldsymbol{\xi}}{dt} &= - P(\bu)^{-*} [D \boldf(t,\bu)]^* P(\bu)^* \boldsymbol{\xi} - P(\bu)^{-*} [DP(\bu)\boldf(t,\bu)]^* \boldsymbol{\xi}.\label{eq:fiberwise-prec-adjoint-system-b}
    \end{align}
\end{subequations}
Comparing the fiberwise-preconditioned adjoint equation \eqref{eq:fiberwise-prec-adjoint-system-b} to the preconditioned adjoint equation \eqref{eq:two-scale-adjoint-DE-prec} using the global transformation discussed in \Cref{sec:induced-from-duality}, we now see an additional term $-P(\bu)^{-*} [DP(\bu) \boldf(t,\bu)]^* \boldsymbol{\xi}$, which arises from the fact that $P$ is not constant but varies with  respect to $\bu$; the term vanishes when $DP = 0$, i.e., the transformation is constant, recovering the previous case discussed in \Cref{sec:induced-from-duality}.

We now prove that this fiberwise-preconditioned adjoint equation satisfies an adjoint-variational conservation law with the state-dependent duality pairing.
\begin{prop}\label{prop:prec-adjoint-conservation}
    Let $(\bu,\boldsymbol{\xi}, \delta\bu)$ be a curve on $T^*U \oplus_U TU$ satisfying the state dynamics, fiberwise-preconditioned adjoint equation, and variational equation, i.e.,
    \begin{subequations}\label{eq:fiberwise-adjoint-variational}
    \begin{align}
        \frac{d}{dt}\bu &= \boldf(t,\bu), \\
         \frac{d}{dt} \boldsymbol{\xi} &= - P(\bu)^{-*} [D \boldf(t,\bu)]^* P(\bu)^* \boldsymbol{\xi} - P(\bu)^{-*} [DP(\bu)\boldf(t,\bu)]^* \boldsymbol{\xi},\\
        \frac{d}{dt}\delta\bu &= D\boldf(t,\bu)\delta\bu.
    \end{align}
    \end{subequations}
    Then, 
    \begin{equation}
        \frac{d}{dt} \langle \boldsymbol{\xi}(t), \delta\bu(t) \rangle_{P(\bu(t))} = 0.
    \end{equation}
    \begin{proof}
        \sloppy This follows from a direct calculation. Note that in the calculation below, an additional term arises from the derivative of the state-dependent map \eqref{eq:fiber-duality-map} in the pairing $\langle \boldsymbol{\xi}(t), \delta\bu(t) \rangle_{P(\bu(t))} = \llangle \boldsymbol{\xi}(t), P(\bu(t)) \delta\bu(t)\rrangle$; however, as we will see, this term is precisely canceled out by the additional term in the fiberwise-preconditioned adjoint equation \eqref{eq:fiberwise-prec-adjoint-system-b}. With the variables $(\bu,\boldsymbol{\xi},\delta\bu)$ below evaluated at time $t$, compute
        \begin{align*}
            \frac{d}{dt} \langle \boldsymbol{\xi},& \delta\bu \rangle_{P(\bu)} = \frac{d}{dt} \llangle \boldsymbol{\xi}, P(\bu) \delta\bu \rrangle \\
            &= \llangle \frac{d\boldsymbol{\xi}}{dt} , P(\bu) \delta\bu\rrangle + \llangle \boldsymbol{\xi}, \frac{dP(\bu)}{dt} \delta\bu\rrangle + \llangle \boldsymbol{\xi}, P(\bu) \frac{d\delta\bu}{dt} \rrangle \\
            &= - \llangle \boldsymbol{\xi}, P(\bu) D\boldf(t,\bu) \delta \bu \rrangle - \llangle \boldsymbol{\xi}, [DP(\bu) \boldf(t,\bu)] \delta \bu \rrangle + \llangle \boldsymbol{\xi}, [DP(\bu) \boldf(t,\bu)] \delta \bu \rrangle \\
            & \qquad + \llangle \boldsymbol{\xi}, P(\bu) D\boldf(t,\bu)\delta\bu \rrangle \\
            &= 0,
        \end{align*}
        where we used that $\frac{dP(\bu)}{dt} = DP(\bu) \dot{\bu} = DP(\bu) \boldf(t,\bu)$.
    \end{proof}
            \begin{remark}
            Alternatively, using a similar argument as done in \cite{TrLe2024}, this proposition also follows from the (extended) symplecticity of the Hamiltonian flow.
        \end{remark}
\end{prop}

We note that the fiberwise-preconditioned adjoint equation \eqref{eq:fiberwise-prec-adjoint-system-b} has a nice interpretation in terms of a connection \cite{KoNo1996} on the vector bundle $T^*U \rightarrow U$. Note that the new term in \eqref{eq:fiberwise-prec-adjoint-system-b} has the coordinate expression 
$$ -\Big(P(\bu)^{-*} [DP(\bu)\boldf(t,\bu)]^* \boldsymbol{\xi} \Big)_\nu = - (P(\bu)^{-1})^\alpha_\nu D_\gamma P^\beta_\alpha(\bu) f^\gamma(t,\bu) \xi_\beta.$$
Define the \emph{Christoffel symbols} (e.g., see section III.7 of \cite{KoNo1996}) corresponding to the fiberwise duality pairing by
$$ \Gamma^\beta_{\nu \gamma}(\bu) := (P(\bu)^{-1})^\alpha_\nu D_\gamma P^\beta_\alpha(\bu). $$
Then, define the total time derivative of a section $\mathbf{s} = (s_\beta)$ of $T^*U$ over a curve $\bu = (u^\alpha)$ of $U$,
\begin{equation}\label{eq:total-time-derivative}
    \left(\frac{D}{Dt} \mathbf{s}\right)_\nu = \frac{ds_\nu}{dt} + \Gamma^\beta_{\nu \gamma}(\bu) \dot{u}^\gamma s_\beta.
\end{equation}
It is straightforward to verify that this corresponds to a vector bundle connection on $T^*U \rightarrow U$, with $d/dt = \dot{u}^\alpha \partial/\partial u^\alpha$. Using $\boldf = \dot{\bu}$ in \eqref{eq:fiberwise-prec-adjoint-system-b}, the fiberwise-preconditioned adjoint equation can be equivalently expressed
\begin{equation}\label{eq:fiberwise-adjoint-eq-connection}
    \frac{D}{Dt} \boldsymbol{\xi} = - P(\bu)^{-*} [D \boldf(t,\bu)]^* P(\bu)^* \boldsymbol{\xi}.
\end{equation}
This gives a nice interpretation of the fiberwise-preconditioned adjoint equation; it has the same form as the preconditioned adjoint equation  \eqref{eq:fiberwise-prec-adjoint-system-b} using a global transformation, but with $d/dt$ replaced with the connection $D/Dt$, which accounts for the duality pairing varying on $T^*U \oplus_U TU$. We can make the expression \eqref{eq:fiberwise-adjoint-eq-connection} further invariant as follows. Recalling that the adjoint of an operator with respect to the pairing $\langle\cdot,\cdot\rangle_{P(\bu)}$, denoted $[\ \cdot\ ]^{*P(\bu)}$, is related to the adjoint with respect to the standard pairing by conjugation, e.g., as appears in the right-hand side of \eqref{eq:fiberwise-adjoint-eq-connection}, this equation can be expressed
\begin{equation}\label{eq:fiberwise-adjoint-eq-connection-invariant}
    \frac{D}{Dt} \boldsymbol{\xi} = -  [D \boldf(t,\bu)]^{*P(\bu)}  \boldsymbol{\xi}.
\end{equation}
This equation has the same form as the canonical adjoint equation \eqref{eq:vector-space-adjoint-b}, with $d/dt$ replaced by $D/Dt$ and the standard adjoint replaced with the adjoint with respect to the fiberwise duality pairing. We can interpret the canonical adjoint equation \eqref{eq:vector-space-adjoint-b} as a special case of \eqref{eq:fiberwise-adjoint-eq-connection-invariant}, noting that the Christoffel symbols associated with the standard duality pairing vanish (i.e., the connection is flat), so that $D/Dt = d/dt$ for the standard duality pairing.

Interestingly, this vector bundle connection has no relation to a Riemannian connection (in this section, we never assumed that $TU$ has a Riemannian structure); it arose by specifying a fiberwise, smoothly varying non-degenerate pairing $\langle \cdot,\cdot\rangle_{P(\bu)}: T_\bu^*U \times T_\bu U \rightarrow \mathbb{R}$, whereas a Riemannian structure is given by a fiberwise, smoothly varying inner product on the tangent spaces $g_\bu: T_\bu U \times T_\bu U \rightarrow \mathbb{R}$. It would be interesting to explore further the geometry of connections associated with fiberwise duality pairings; for example, to formulate such connections on general curved manifolds and explore related concepts such as curvature and torsion. In particular, the invariant expression \eqref{eq:fiberwise-adjoint-eq-connection-invariant} suggests that this preconditioning transformation can be similarly formulated for adjoint systems with state dynamics evolving on a manifold.

\subsection{Discussion}\label{sec:discussion}
Before proceeding to applications of this framework in \Cref{sec:adjoint-systems-coupled}, we provide here some perspectives on the adjoint preconditioning developed thus far. In particular, we introduced two types of adjoint preconditioning: those induced from a transformation of the state dynamics, and those induced from a (possibly state-dependent) duality pairing.  Both types of adjoint preconditioning can in principle be applied to an adjoint system and furthermore, serve different functional purposes. 

On one hand, the adjoint preconditioning induced from a transformation of the state dynamics can be interpreted as providing the natural representation of the adjoint dynamics, given a new representation of the state dynamics specified by the transformation. A common example where such a change in representation occurs is when working with semi-discretized evolutionary PDEs with mass matrices. For state-independent mass matrices, the corresponding representation of the adjoint equation is discussed in detail in \cite{TrSoLe2024}; however, the transformations presented here are more general, in that they allow for nonlinear transformations of the state dynamics which include state-dependent mass matrices which arise, for example, in Lagrangian hydrodynamics \cite{DoKoRi2012}.

On the other hand, adjoint preconditioning induced by a duality pairing can be interpreted as transforming only the adjoint variable (and hence, changing only the ``backpropagation" dynamics) while keeping the state dynamics invariant. This form of preconditioning is useful when one only wants to reshape the gradient dynamics, for example, in order to stabilize or accelerate adjoint-based gradient optimization, without affecting the state dynamics.

\textbf{Convergence.} In general, convergence of adjoint-based gradient optimization will depend on the problem. For example, nonlinearity of dynamics can lead to loss of convexity of the optimization problem. As such, we will here sketch a proof of the convergence in the simple case of an \emph{affine} flow, where convexity of the cost is preserved under the dynamics.

Consider a cost function of the terminal state
$ C(\bu(T)) $
and consider minimizing this cost with respect to the initial condition $\bu(0) = \bu_0$ for the dynamics $\dot{\bu}=\boldf(t,\bu)$. This is a constrained optimization problem with the state dynamics as the constraint. Assume the time $0$ to time $T$ flow of $\boldf$ exists and denote it by $\Phi^T_0$. Then, we can formally express the constrained optimization problem as an unconstrained optimization problem $\min_{\bu_0} \widetilde{C}(\bu_0)$ for a new cost function defined by
$$ \widetilde{C}(\bu_0) := C(\Phi^T_0(\bu_0)). $$
This is equivalent to the original constrained optimization problem since $\bu(T) = \Phi^T_0(\bu_0)$. Now, suppose that $C$ is convex. As mentioned above, for general flows, even if $C$ is convex, $\widetilde{C}$ may not be. However, assuming that the flow is affine, corresponding to state dynamics $\dot{\bu} = A(t)\bu + \mathbf{b}(t)$ where $A$ is a time-dependent linear operator and $\mathbf{b}$ is a time-dependent vector, then $\Phi^T_0$ is an affine transformation, which implies that $\widetilde{C}$ is convex since convexity is preserved under affine transformations \cite{nesterov_lectures}. The adjoint method is constructed such that the adjoint variable at the initial time is precisely the derivative of the new cost $\widetilde{C}(\bu_0)$, which follows from the adjoint-variational quadratic conservation law (\Cref{prop:ManifoldQuadraticInvariant}). The key to proving convergence of gradient descent for the preconditioned adjoint framework is that the preconditioned adjoints also enjoy adjoint-variational quadratic conservation laws (\Cref{prop:prec-1} and \Cref{prop:prec-adjoint-conservation}), and thus also yield the derivative of the unconstrained cost, possibly with respect to a modified duality pairing. Then, using the convexity of $\widetilde{C}(\bu_0)$, one could use standard arguments for gradient descent \cite{nesterov_lectures} to prove convergence in the simplified setting of affine flows. Note that, in the standard convergence theory of gradient descent, the speed of convergence depends on the condition number of the Hessian of the objective (with higher condition number leading to slower convergence). In light of this, in subsequent work, we plan to explore the effect of adjoint preconditioning on the speed of convergence.

\textbf{Computational cost.} For state-dependent adjoint preconditioning, the additional computational cost will depend on how efficiently one can apply the adjoint preconditioner and integrate the resulting preconditioned adjoint equation \eqref{eq:fiberwise-prec-adjoint-system-b}. In general, this is highly dependent on the problem and choice of adjoint-preconditioning. The example we consider in \Cref{sec:numerical} uses a diagonal preconditioning, which introduces a negligible computational cost. More complex (non diagonal) scale preconditionings could come with a larger overhead in computational cost, but this would depend on whether one has efficient ways to apply the adjoint preconditioner (e.g., via preconditioned Krylov methods), as well as the \emph{relative} cost and complexity of the adjoint preconditioning compared with the base cost of integrating the forward and adjoint equations. However, we note that the goal of the preconditioned adjoint framework is, despite potentially adding additional computational cost per backward solve of the adjoint equation, to reduce the overall cost of the adjoint-based optimization, by accelerating convergence of the descent and thus reducing the total number of forward and backward solves. 

Although we will focus only on adjoint preconditioning to resolve scale preconditioning in the subsequent section and numerical example, we note that extending this to the general state-dependent setting involves two non-trivial difficulties beyond what is explored here. First, a state-independent preconditioner $P$, as used in the numerical example, causes the connection term $DP(\bu)$ in \eqref{eq:fiberwise-prec-adjoint-system-b} to vanish identically; on the other hand, a state-dependent preconditioner requires forming and applying $DP(\bu)$ at every time step of the adjoint solve. Second, choosing $P(\bu)$ to improve the conditioning of adjoint-based optimization requires a principled strategy; in the setting of coupled equations explored below, such a choice is suggested by the physical scale separation between the two physical variables (energy and temperature), which is similar (in an integrated sense) across the full simulation. General systems may have scale separation that changes with time and state, have scale separation that is inherent to dynamics and length scales rather than solution magnitudes, or benefit from a completely different form of preconditioning. Treating this topic appropriately requires more care that is beyond the scope of this paper, and is a topic for future work.

\textbf{Target problems:} Adjoint-based optimization is most relevant when the dimension $N$ of the state dynamics is large in comparison to the number of cost functions $M$ (typically $M=1,2$) under consideration. In particular, adjoint methods add a fixed $\mathcal{O}(M)$ adjoint equation solves per optimization iteration. This is in contrast to, e.g. forward sensitivity analysis, which requires $\mathcal{O}(N)$ linearized forward equation solves per optimization iteration. As a result, adjoint methods are particularly relevant for PDE-constrained optimization, where the dimension of the (semi-discretized) state dynamics is easily in the hundreds of thousands or millions. The preconditioned adjoint framework preserves the $\mathcal{O}(M)$ adjoint equation solves per optimization iteration scaling, despite the possibility that each solve of the adjoint equation may be more expensive due to the adjoint preconditioning. As such, we expect this framework to be scalable to high-dimensional PDE-constrained optimization when the adjoint preconditioners are constructed to be efficient to update/apply and accelerate the descent. More broadly, discrete adjoint methods are mathematically equivalent in terms of evaluating gradients to reverse-mode automatic differentiation, i.e. backpropagation, which is widely used in machine learning. If and how the framework proposed here can be wrapped into backpropagation in the context of machine learning is an interesting question and topic for future work, but the high dimensionality of machine learning optimization problems is not a fundamental limitation. 

We will now apply this framework in the setting of coupled evolution equations in \Cref{sec:adjoint-systems-coupled}, particularly for coupled evolution equations which admit large scale separation. This application and the accompanying numerical example (\Cref{sec:numerical}) will demonstrate how adjoint preconditioning induced from a duality pairing can be used to adjust adjoint-based gradient optimization to account for such large scale separation in the state dynamics. However, note that the framework introduced in the preceding sections is more general and applies beyond coupled evolution equations; for example, in principle one can decompose a (non-coupled) evolution equation by dynamical scales (e.g., spatial scales, temporal scales, eigenmodes, or magnitudes) and develop adjoint preconditioning based on such a decomposition. Although we will only focus on adjoint preconditioning to resolve scale separation in coupled evolution equations in the subsequent sections, we plan to explore more general strategies and approaches for adjoint preconditioning in future work.

\section{Adjoint systems for coupled evolution equations}\label{sec:adjoint-systems-coupled}
To apply the preconditioned adjoint formulations, we turn our attention to adjoint systems for coupled evolution equations. We will begin by discussing duality on a Cartesian product of vector spaces in \Cref{sec:duality-cartesian}. Then, in \Cref{sec:adjoint-coupled-evolution}, we introduce coupled evolution equations on a Cartesian product and discuss the adjoint equation corresponding to coupled evolution equations.

We will then discuss applications of the adjoint preconditioning developed in \Cref{sec:adjoint-preconditioning-general} in the context of coupled evolution equations. Particularly, as an application of adjoint preconditioning induced by a duality pairing, we discuss scale preconditioning in \Cref{sec:scale-preconditoning-coupled}. We discuss some practical considerations for time integration of the adjoint system in \Cref{sec:time-integration}. We then utilize this scale preconditioner and the adjoint of evolution equations involving mass matrices (see \Cref{ex:mass-matrices}) in the numerical example, \Cref{sec:numerical}.

\subsection{Duality for a Cartesian product}\label{sec:duality-cartesian}
 For the remainder, we will be interested in state dynamics given by coupled evolution equations and their adjoints. As such, it is useful to specialize the discussion of \Cref{sec:duality} to the case when $U$ decomposes as a Cartesian product of vector spaces $U = X \times Y$. Let $((\cdot,\cdot))^X$ and $((\cdot,\cdot))^Y$ be inner products on $X$ and $Y$, respectively. This defines an inner product on $U$ by: for any $\bu_1 = (\bx_1, \by_1) \in U$ and $\bu_2 = (\bx_2, \by_2) \in U$,
$$ ((\bu_1,\bu_2)) := ((\bx_1,\bx_2))^X + ((\by_1,\by_2))^Y. $$
Similarly, we denote the standard duality pairings as $\llangle\cdot,\cdot\rrangle^X: X^* \times X \rightarrow \mathbb{R}$ and $\llangle\cdot,\cdot\rrangle^Y: Y^* \times Y \rightarrow \mathbb{R}$. Viewing $U^* = X^* \times Y^*$, a linear functional $\bp \in U^*$ can be decomposed $\bp = (\bp_\bx, \bp_\by)$, such that its action on $ \bu = (\bx,\by) \in U$ decomposes as
$$ \llangle \bp,\bu\rrangle = \bp(\bu) = \bp_\bx(\bx) + \bp_\by(\by) = \llangle \bp_\bx, \bx\rrangle^X + \llangle \bp_\by, \by\rrangle^Y.$$

Note that the derivative of a map $C: U = X \times Y \rightarrow \mathbb{R}$ can be decomposed into its components,
$$ U^* \ni DC(\bu) = \begin{pmatrix} D_\bx C(\bx,\by) \\ D_\by C(\bx,\by) \end{pmatrix}, $$
where $\bu = (\bx,\by)$, $D_\bx C(\bx,\by) \in X^*$ and $D_\by C(\bx,\by) \in Y^*$. Throughout, we will also express vectors in a Cartesian product vertically, e.g., for $\bu = (\bx,\by) \in X \times Y$ and $\bp = (\bp_\bx, \bp_\by) \in X^* \times Y^*$, we express
$$ \bu = \begin{pmatrix} \bx \\ \by \end{pmatrix} \text{ and } \bp = \begin{pmatrix} \bp_\bx \\ \bp_\by \end{pmatrix}.$$

Similarly, the derivative of a map $\boldf: U = X \times Y \rightarrow U = X \times Y$ with components $\boldf: (\bx,\by) \mapsto (\boldf_\bx(\bx,\by), \boldf_\by(\bx,\by)) \in X \times Y$, can be decomposed in \emph{block form} as
\begin{equation}\label{eq:Df}
    L(U,U) \ni D\boldf(\bu) = \begin{pmatrix} D_\bx \boldf_\bx (\bx,\by) & D_\by \boldf_\bx(\bx,\by)  \\ D_\bx \boldf_\by(\bx,\by)  & D_\by \boldf_\by(\bx,\by) \end{pmatrix},
\end{equation}
where the blocks are mappings
\begin{align*}
    D_\bx \boldf_\bx(\bx,\by) \in L(X,X),  \quad
    D_\by \boldf_\bx(\bx,\by) \in L(Y,X),  \quad
    D_\bx \boldf_\by(\bx,\by) \in L(X,Y),  \quad
    D_\by \boldf_\by(\bx,\by) \in L(Y,Y).
\end{align*}
The adjoint of this operator with respect to the standard pairing can then be expressed in block form as
\begin{equation}\label{eq:Df-adjoint}
    L(U^*,U^*) \ni [D\boldf(\bu)]^* = \begin{pmatrix} [D_\bx \boldf_\bx (\bx,\by)]^* & [D_\bx \boldf_\by(\bx,\by)]^*  \\ [D_\by \boldf_\bx(\bx,\by)]^*  & [D_\by \boldf_\by(\bx,\by)]^* \end{pmatrix},
\end{equation}
where the blocks are defined by
\begin{align*}
    L(X^*,X^*) \ni [D_\bx \boldf_\bx(\bx,\by)]^* &\text{ such that } \llangle\, [D_\bx \boldf_\bx(\bx,\by)]^* \bp_\bx, \bv_\bx \rrangle^X = \llangle \bp_\bx, D_\bx \boldf_\bx(\bx,\by) \bv_\bx\rrangle^X,  \\
    L(Y^*,X^*) \ni [D_\bx \boldf_\by(\bx,\by)]^* &\text{ such that } \llangle\, [D_\bx \boldf_\by(\bx,\by)]^* \bp_\by, \bv_\bx \rrangle^X = \llangle \bp_\by, D_\bx \boldf_\by(\bx,\by) \bv_\bx\rrangle^Y,  \\
    L(X^*,Y^*) \ni [D_\by \boldf_\bx(\bx,\by)]^* &\text{ such that } \llangle\, [D_\by \boldf_\bx(\bx,\by)]^* \bp_\bx, \bv_\by \rrangle^Y = \llangle \bp_\bx, D_\by \boldf_\bx(\bx,\by) \bv_\by\rrangle^X,  \\
    L(Y^*,Y^*) \ni [D_\by \boldf_\by(\bx,\by)]^* &\text{ such that } \llangle\, [D_\by \boldf_\by(\bx,\by)]^* \bp_\by, \bv_\by \rrangle^Y = \llangle \bp_\by, D_\by \boldf_\by(\bx,\by) \bv_\by\rrangle^Y,
\end{align*}
for all $\bv_\bx \in X$, $\bv_\by \in Y$, $\bp_\bx \in X^*$, $\bp_\by \in Y^*$.

\subsection{Adjoint of coupled evolution equations}\label{sec:adjoint-coupled-evolution}
We now discuss coupled evolution equations and their adjoints. For concreteness, we consider systems of two coupled evolution equations. However, in principle, the discussion of this section follows similarly for any number of coupled evolution equations with the obvious modifications.

As before, let $U$ be a finite-dimensional vector space that decomposes as the product of vector spaces $U = X \times Y$. We consider state dynamics given by \emph{coupled evolution equations} on $X \times Y \ni (\bx,\by)$, on a time interval $[0,T]$,
\begin{subequations}\label{eq:coupled-evolution-equation}
    \begin{align}
        \frac{d}{dt} \begin{pmatrix} \bx \\ \by \end{pmatrix} &= \boldf(t,\bx,\by), \label{eq:coupled-evolution-equation-a}\\
        \bx(0) &= \bx_0, \\
        \by(0) &= \by_0,
    \end{align}
\end{subequations}
where $\boldf$ is a time-dependent vector field on $X \times Y$, i.e., $\boldf: [0,T] \times X \times Y \rightarrow X \times Y$. We denote the components of $\boldf$ valued in $X$ and $Y$ as $\boldf_\bx$ and $\boldf_\by$, respectively. The evolution equation \eqref{eq:coupled-evolution-equation-a} is then expressed as
\begin{align*}
    \frac{d}{dt} \begin{pmatrix} \bx \\ \by \end{pmatrix} &= \begin{pmatrix} \boldf_\bx(t,\bx,\by) \\ \boldf_\by(t,\bx,\by) \end{pmatrix}.
\end{align*}

The adjoint equation on $X^* \times Y^* \ni (\bp_\bx,\bp_\by)$ associated to the coupled evolution equation \eqref{eq:coupled-evolution-equation} can be expressed in block form as
\begin{equation}\label{eq:adjoint-coupled-evolution-equation}
   \frac{d}{dt} \begin{pmatrix} \bp_\bx \\ \bp_\by \end{pmatrix} = - \begin{pmatrix} [D_\bx \boldf_\bx (t,\bx,\by)]^* & [D_\bx \boldf_\by(t,\bx,\by)]^*  \\ [D_\by \boldf_\bx(t,\bx,\by)]^*  & [D_\by \boldf_\by(t,\bx,\by)]^* \end{pmatrix} \begin{pmatrix} \bp_\bx \\ \bp_\by \end{pmatrix}.
\end{equation}
Note that the state dynamics \eqref{eq:coupled-evolution-equation-a}, together with the adjoint equation \eqref{eq:adjoint-coupled-evolution-equation}, can be regarded as a (time-dependent) Hamiltonian system on $T^*U = T^*(X \times Y) = T^*X \times T^*Y = (X \times X^*) \times (Y \times Y^*)$, with Hamiltonian given by
$$ H(t,\bx,\by,\bp_\bx,\bp_\by) = \llangle \bp_\bx, \boldf_\bx(t,\bx,\by)\rrangle^X + \llangle \bp_\by, \boldf_\by(t,\bx,\by)\rrangle^Y. $$
and canonical symplectic form $\Omega = -\llangle d\bp_\bx \wedge d\bx \rrangle^X - \llangle d\bp_\by \wedge d\by \rrangle^Y$. Such adjoint systems associated to a coupled evolution equation admit the following quadratic conservation law. Let $(\delta \bx,\delta \by)$ satisfy the variational equation given by the linearization of the state dynamics \eqref{eq:coupled-evolution-equation},
\begin{equation}\label{eq:variational-coupled-evolution-equation}
    \frac{d}{dt} \begin{pmatrix} \delta\bx \\ \delta\by \end{pmatrix} = \begin{pmatrix} D_\bx \boldf_\bx (t,\bx,\by) & D_\by \boldf_\bx(t,\bx,\by)  \\ D_\bx \boldf_\by(t,\bx,\by)  & D_\by \boldf_\by(t,\bx,\by) \end{pmatrix}\begin{pmatrix} \delta\bx \\ \delta\by \end{pmatrix}.
\end{equation}
Then, we have following adjoint-variational conservation law for coupled evolution equations.
\begin{prop}\label{prop:adjoint-variational-cons-coupled}
    Let $(\bx,\by,\bp_\bx,\bp_\by,\delta \bx,\delta\by)$ satisfy the state dynamics \eqref{eq:coupled-evolution-equation}, adjoint equation \eqref{eq:adjoint-coupled-evolution-equation}, and variational equation \eqref{eq:variational-coupled-evolution-equation}. Then,
    \begin{equation}\label{eq:adjoint-variational-cons-coupled}
        \frac{d}{dt} \Big(\llangle \bp_\bx, \delta \bx\rrangle^X + \llangle \bp_\by, \delta \by\rrangle^Y\Big) = 0.
    \end{equation}
    \begin{proof}
        This follows from the general case, \Cref{prop:ManifoldQuadraticInvariant}, and using the expressions for the decomposition $U = X \times Y$ discussed in \Cref{sec:duality-cartesian}.
    \end{proof}
\end{prop}

Since such a quadratic conservation law is satisfied, \eqref{eq:adjoint-variational-cons-coupled} implies that $\bp_\bx$ backpropagates derivatives of a cost function with respect to $\bx$, and similarly, $\bp_\by$ backpropagates derivatives of a cost function with respect to $\by$. 

\subsection{Scale preconditioner for coupled evolution equations}\label{sec:scale-preconditoning-coupled}
As discussed in \Cref{sec:induced-from-duality}, one can use a base space-preserving adjoint preconditioning transformation to reshape the derivative obtained through adjoint backpropagation. As an example of such reshaping, we consider adjoint preconditioning when the state dynamics admit large scale-separation, i.e., the state vectors $\bx$ and $\by$ have differing orders of magnitude. As a simple example, it is instructive to consider a block two-by-two system whose blocks are multiples of the identity. 

\begin{example}[Scale preconditioning for the adjoint of a block two-by-two system]\label{ex:scale-preconditioner-block}
    We consider a block two-by-two system of equations such that the first component $\vec{x}$ scales like $|\alpha| \gg 1$ and the second component $\vec{y}$ scales like $\mathcal{O}(1)$. Particularly, consider the state dynamics given by the ODE
    \begin{equation}\label{eq:two-scale-DE}
        \frac{d}{dt} \begin{pmatrix} \vec{x} \\ \vec{y} \end{pmatrix} = \begin{pmatrix} I & \alpha I \\ \beta I & I \end{pmatrix} \begin{pmatrix} \vec{x} \\ \vec{y} \end{pmatrix},
    \end{equation}
    where $\vec{x}(t), \vec{y}(t) \in \mathbb{R}^N$, $I$ denotes the $N \times N$ identity, and $\alpha,\beta \in \mathbb{R}$ are constants such that $|\alpha| \gg 1$ and $0<|\beta| \sim \mathcal{O}(1/|\alpha|)$. The corresponding adjoint equation is given by
    \begin{equation}\label{eq:two-scale-DE-adjoint}
        \frac{d}{dt} \begin{pmatrix} \vec{p}_x \\ \vec{p}_y \end{pmatrix} = -\begin{pmatrix} I & \beta I \\ \alpha I & I \end{pmatrix} \begin{pmatrix} \vec{p}_x \\ \vec{p}_y \end{pmatrix}.
    \end{equation}
    Consider initial conditions $(\vec{x}(0),\vec{y}(0))$ such that $\|\vec{x}(0)\| \sim \mathcal{O}(\alpha) \gg 1$ and $\|\vec{y}(0)\| \sim \mathcal{O}(1)$. Under these assumptions, both terms in the $\vec{x}$-component of the right-hand side of the ODE, $\vec{x} + \alpha \vec{y}$, scale like $\mathcal{O}(\alpha)$; similarly, both terms in the $\vec{y}$-component, $\beta \vec{x} + \vec{y}$, scale like $\mathcal{O}(1)$. 

    Let us further suppose that the adjoint equation is supplied with terminal conditions which scale analogously to their corresponding state variables, i.e., $\vec{p}_x(T) \sim \mathcal{O}(\alpha)$ and $\vec{p}_y(T) \sim \mathcal{O}(1)$. Then, the adjoint equation \eqref{eq:two-scale-DE-adjoint} does not retain the scaling of the state dynamics \eqref{eq:two-scale-DE}. Namely, in the $\vec{p}_x$-component of the adjoint equation, $-\vec{p}_x- \beta \vec{p}_y$, the first term scales like $\mathcal{O}(\alpha)$, whereas the second term scales like $\mathcal{O}(1)$. In the $\vec{p}_y$-component of the adjoint equation, $-\alpha \vec{p}_x- \vec{p}_y$, the first term scales like $\mathcal{O}(\alpha^2)$ and the second term scales like $\mathcal{O}(1)$. This arises from the fact that the block operator appearing in the adjoint equation is (minus) the transpose of the block operator of the state dynamics; thus, the off-diagonal blocks get swapped. 

    For an adjoint-based optimization method, this leads to two distinct issues. A minor issue is that the $\vec{p}_y$-component of the adjoint equation scales as $\mathcal{O}(\alpha^2)$, which is much larger than the scaling of $\vec{y} \sim \mathcal{O}(1)$. This would lead to unstable gradient descent with an $\mathcal{O}(1)$ step-size but is simply resolved by taking an $\mathcal{O}(1/\alpha^2)$ step-size in the $\vec{y}$ direction. The more critical issue is that, in each component of the adjoint equation, the two terms have different scalings. Namely, the $\mathcal{O}(\alpha)$ term in the $\vec{p}_x$-component dominates the behavior, $\|\vec{p}_x\| \gg \|\beta\vec{p}_y\|$, and similarly, the $\mathcal{O}(\alpha^2)$ term in the $\vec{p}_y$-component dominates the behavior, $\|\alpha\vec{p}_x\| \gg \|\vec{p}_y\|$. This means that $\vec{p}_x$ dominates the behavior of the adjoint equation and the dependence on $\vec{p}_y$ decouples as $\alpha \rightarrow \infty$.

    To resolve this improper scaling, we instead use a preconditioned adjoint system induced from a duality pairing as described in \Cref{sec:induced-from-duality},
        \begin{equation}\label{eq:two-scale-DE-similarity}
        \frac{d}{dt} \begin{pmatrix} \vec{\xi}_x \\ \vec{\xi}_y \end{pmatrix} = -S^{-T} \begin{pmatrix} I & \beta I \\ \alpha I & I \end{pmatrix} S^T \begin{pmatrix} \vec{\xi}_x\\ \vec{\xi}_y \end{pmatrix}.
    \end{equation}
    In particular, choosing 
    $$ S = \begin{pmatrix}
        I & 0 \\ 0 & \frac{\alpha}{\beta}
    \end{pmatrix},$$
    equation \eqref{eq:two-scale-DE-similarity} becomes
    \begin{equation}
        \frac{d}{dt} \begin{pmatrix} \vec{\xi}_x\\ \vec{\xi}_y \end{pmatrix} = -\begin{pmatrix} I & \alpha I \\ \beta I & I \end{pmatrix} \begin{pmatrix} \vec{\xi}_x\\ \vec{\xi}_{y}\end{pmatrix}.
    \end{equation}
    The transformed adjoint system has the same scaling properties as the state dynamics. We refer to $S$ as a scale preconditioned duality pairing and the value $\alpha/\beta$ as the scale preconditioner since it rescales the adjoint equation to have the same scaling of the state dynamics.\qed
\end{example}

We will generalize the previous example to consider state dynamics for $(\vec{x}(t),\vec{y}(t)) \in \mathbb{R}^N \times \mathbb{R}^N$ given by a linear system, 
$$ \frac{d}{dt} \begin{pmatrix} \vec{x} \\ \vec{y} \end{pmatrix} = \underbrace{\begin{pmatrix} A & B \\ C & D \end{pmatrix}}_{=:\ L} \begin{pmatrix} \vec{x} \\ \vec{y} \end{pmatrix}, $$
where the blocks $A$, $B$, $C$, and $D$ are symmetric. We will proceed informally in this discussion, but roughly, when we say an operator $M \sim \mathcal{O}(\nu)$, we mean that it maps a given vector $\vec{v} \sim \mathcal{O}(\mu)$ to a vector $M \vec{v} \sim \mathcal{O}(\mu\nu)$. Then, we further assume the blocks of $L$ scale as $A \sim \mathcal{O}(1)$, $B \sim \mathcal{O}(\beta)$, $C \sim \mathcal{O}(\alpha)$, and $D \sim \mathcal{O}(1)$ where, as in \Cref{ex:scale-preconditioner-block}, $\alpha \gg 1$ and $\beta \ll 1$. This can be visualized schematically as follows,
\begin{equation}\label{eq:forward-map-schematic}
    L = \left( \begin{tikzcd}[sep=small]
	&&&& {\mathcal{O}(1)} \\
	{\mathcal{O}(1)} \\
	&&&& {\mathcal{O}(\beta)} \\
	{\mathcal{O}(\alpha)} \\
	&&&& {\mathcal{O}(1)} \\
	{\mathcal{O}(1)}
	\arrow["A", from=2-1, to=2-1, loop, in=55, out=125, distance=10mm]
	\arrow["B", from=1-5, to=3-5]
	\arrow["C"', from=6-1, to=4-1]
	\arrow["D"', from=5-5, to=5-5, loop, in=305, out=235, distance=10mm]
\end{tikzcd} \right),
\end{equation}
where we visualize a downward arrow as an operator lowering the magnitude and an upward arrow as an operator increasing the magnitude. Now, for the operator $L^*$ appearing in the adjoint equation, we would like it to act in the same way as $L$, as discussed in \Cref{ex:scale-preconditioner-block}. This is not the case, even with the assumption that the blocks are each symmetric, since the off-diagonal blocks are swapped under the adjoint operation, i.e.,
\begin{equation}\label{eq:backward-map-schematic}
    L^* = \left( \begin{tikzcd}[sep=small]
	&&&& {\mathcal{O}(\alpha)} \\
	{\mathcal{O}(1)} \\
	&&&& {\mathcal{O}(1)} \\
	{\mathcal{O}(1)} \\
	&&&& {\mathcal{O}(1)} \\
	{\mathcal{O}(\beta)}
	\arrow["A", from=2-1, to=2-1, loop, in=55, out=125, distance=10mm]
	\arrow["B", from=4-1, to=6-1]
	\arrow["C"', from=3-5, to=1-5]
	\arrow["D"', from=5-5, to=5-5, loop, in=305, out=235, distance=10mm]
\end{tikzcd} \right).
\end{equation}
Instead of considering the operator $L^*$ corresponding to the standard duality pairing, we can consider the operator $P^{-*} L^* P^*$ from the preconditioned adjoint equation with respect to a duality pairing, as discussed in \Cref{sec:induced-from-duality}. Specifically, $P$ is an invertible linear transformation on $\mathbb{R}^N \times \mathbb{R}^N$. We have the freedom to choose $P$ so long as it is invertible. In analogy with \Cref{ex:scale-preconditioner-block}, we consider a block diagonal form, 
$$ P = \begin{pmatrix} P_1 & 0 \\ 0 & P_2 \end{pmatrix}, $$
where $P_1$ and $P_2$ are symmetric invertible linear transformations on $\mathbb{R}^N$. The preconditioned adjoint operator is given by
$$ P^{-1}L^*P = \begin{pmatrix} P_1^{-1}AP_1 & -P_1^{-1}CP_2 \\ -P_2^{-1}BP_1 & P_2^{-1}DP_2 \end{pmatrix}. $$
Now, we choose the scaling of $P_1$ and $P_2$ so that $P^{-1}L^*P$ acts in the same way as $L$. A straight-forward calculation shows that if we choose scalings $P_1 \sim \mathcal{O}(\sigma)$ and $P_2 \sim \mathcal{O}(\nu)$, where $\nu/\sigma = \alpha/\beta$, with the assumption that $P_1^{-1} \sim \mathcal{O}(\sigma^{-1})$ and $P_2^{-1} \sim \mathcal{O}(\nu^{-1})$, then the preconditioned adjoint operator $P^{-1}L^*P$ acts in a similar manner to $L$ \eqref{eq:forward-map-schematic}, i.e.,
\begin{equation} \label{eq:transformed-adjoint-map-schematic}
    P^{-1}L^*P = \left( \begin{tikzcd}[sep=small]
	&&&& {\mathcal{O}(1)} \\
	{\mathcal{O}(1)} \\
	&&&& {\mathcal{O}(\beta)} \\
	{\mathcal{O}(\alpha)} \\
	&&&& {\mathcal{O}(1)} \\
	{\mathcal{O}(1)}
	\arrow["P_1^{-1}AP_1", from=2-1, to=2-1, loop, in=55, out=125, distance=10mm]
	\arrow["P_1^{-1}CP_2", from=1-5, to=3-5]
	\arrow["P_2^{-1}BP_1"', from=6-1, to=4-1]
	\arrow["P_2^{-1}DP_2"', from=5-5, to=5-5, loop, in=305, out=235, distance=10mm]
\end{tikzcd} \right) 
\end{equation}

\begin{remark}
    Of course, the simplest choice of $P_1$ and $P_2$ are multiples of the identity, as discussed in \Cref{ex:scale-preconditioner-block}. We will use this choice in the numerical example, \Cref{sec:numerical}. However, the theory developed in \Cref{sec:induced-from-duality} and \Cref{sec:fiberwise-duality} is more general, including the possibility of preconditioners that are not constant multiples of the identity (e.g., to treat problems with multiple spatial scales) or preconditioners that are state-dependent (for example, to develop a robust, state-adapted adjoint backpropagation). We plan to explore these possibilities in future work. Roughly speaking, when the scales of the dynamics are not known \emph{a priori}, a state-dependent scale preconditioner should be chosen to scale inversely proportionally to $\bu$, either component-wise or in some appropriate norm.
\end{remark}

\subsection{Time integration}\label{sec:time-integration} As a practical matter, we provide some details regarding the time integration of the state dynamics \eqref{eq:coupled-evolution-equation} and its adjoint \eqref{eq:adjoint-coupled-evolution-equation}. For the state dynamics, we consider a semi-implicit time integration scheme, where $\boldf(t,\bx,\by)$ may generally depend nonlinearly on $\bx$ and $\by$. We decompose $\boldf: [0,T] \times X \times Y \rightarrow X \times Y$ using a \emph{partitioned vector field} \cite{nprk1, nprk2}, which is a mapping 
$$\boldF: \Big([0,T] \times X \times Y\Big) \times \Big([0,T] \times X \times Y\Big) \rightarrow X \times Y$$
that agrees with $\boldf$ along the diagonal, i.e.,
\begin{equation}\label{eq:partitioned-vector-field}
    \boldF(t, \bx, \by; t, \bx, \by) = \boldf(t, \bx,\by).
\end{equation}
Given a partitioned vector field $\boldF$ of $\boldf$, the semi-implicit Euler scheme for the state dynamics \eqref{eq:coupled-evolution-equation} with timestep $\Delta t$ is given by
\begin{align}\label{eq:coupled-semi-imp-euler}
    \begin{pmatrix} \bx_{n+1} \\ \by_{n+1} \end{pmatrix}  &= \begin{pmatrix} \bx_n \\ \by_n \end{pmatrix} + \Delta t\, \boldF(t_n, \bx_n, \by_n; t_{n+1}, \bx_{n+1}, \by_{n+1}) \\
    &= \begin{pmatrix} \bx_n \\ \by_n \end{pmatrix} + \Delta t \begin{pmatrix} \boldF_{\bx}(t_n, \bx_n, \by_n; t_{n+1}, \bx_{n+1}, \by_{n+1}) \\ \boldF_{\by}(t_n, \bx_n, \by_n; t_{n+1}, \bx_{n+1}, \by_{n+1})\end{pmatrix}, \nonumber
\end{align}
where we have denoted the components of $\boldF$ valued in $X$ and $Y$ as $\boldF_\bx$ and $\boldF_\by$, respectively. For concreteness, we consider the first-order semi-implicit Euler scheme, although higher-order semi-implicit schemes, such as semi-implicit Runge--Kutta schemes \cite{Boscarino.2016,Boscarino.2023} or the recently introduced nonlinearly partitioned Runge--Kutta schemes \cite{nprk1, nprk2, nprk-mr}, could be treated similarly.

In \eqref{eq:partitioned-vector-field}, we have split the dependence of $\boldf$ on the state variables into two arguments; in the semi-implicit Euler scheme \eqref{eq:coupled-semi-imp-euler}, the first argument is treated explicitly while the second argument is treated implicitly. As such, a partitioned vector field is usually chosen so that $\boldF(t_1, \bx_1, \by_1; t_2, \bx_2, \by_2)$ has non-stiff dependence on $(t_1, \bx_1, \by_1)$ and stiff dependence on $(t_2, \bx_2, \by_2)$. This is a fairly general decomposition of the state dynamics since it includes additive implicit-explicit (IMEX) splittings \cite{Ascher97, Kennedy.2003tv4} and, more generally, nonlinear IMEX splittings \cite{Boscarino.2016,Boscarino.2023, Boscarino.2015, nprk1, nprk-mr}.

\sloppy For a fixed timestep $t_n \mapsto t_{n+1}$, since $(\bx_n,\by_n)$ is known, \eqref{eq:coupled-semi-imp-euler} can be regarded as a generally nonlinear equation $G(\bx_{n+1}, \by_{n+1}) = 0$ in the unknowns $(\bx_{n+1},\by_{n+1})$, where
\begin{align}\label{eq:newton-form-F}
    G(\bx_{n+1}, \by_{n+1}) &:=  \begin{pmatrix} \bx_{n+1} \\ \by_{n+1} \end{pmatrix} - \begin{pmatrix} \bx_n \\ \by_n \end{pmatrix} - \Delta t \begin{pmatrix} F_{\bx}(t_n, \bx_n, \by_n; t_{n+1}, \bx_{n+1}, \by_{n+1}) \\ F_{\by}(t_n, \bx_n, \by_n; t_{n+1}, \bx_{n+1}, \by_{n+1})\end{pmatrix}
\end{align}
The nonlinear equation $G(\bx_{n+1}, \by_{n+1}) = 0$ can be iteratively solved by successive linearization. Here, we consider applying Newton's method to this equation: denoting the $k^{th}$ iterate of Newton's method as $(\bx_{n+1}^k, \by_{n+1}^k)$ with initial iterate $(\bx_{n+1}^0, \by_{n+1}^0) = (\bx_{n}, \by_{n})$, Newton's method applied to \eqref{eq:newton-form-F} is given by
\begin{equation}\label{eq:newton-method}
    \begin{pmatrix}
        N_\bxx & N_\bxy \\ N_\byx & N_\byy
    \end{pmatrix} \left(  \begin{pmatrix} \bx_{n+1}^{k+1} \\ \bx_{n+1}^{k+1}  \end{pmatrix} - \begin{pmatrix} \bx_{n+1}^{k} \\ \bx_{n+1}^{k}  \end{pmatrix}  \right) = - G(\bx_{n+1}^k, \by_{n+1}^k),
\end{equation}
where the blocks are defined as
\begin{align*}
        N_\bxx &:= I_X - \Delta t D_{\bx_{n+1}} \boldF_\bx(t_n, \bx_n, \by_n; t_{n+1}, \bx_{n+1}^k, \by_{n+1}^k) , \\
        N_\bxy &:= -\Delta t D_{\by_{n+1}} \boldF_\bx(t_n, \bx_n, \by_n; t_{n+1}, \bx_{n+1}^k, \by_{n+1}^k), \\
        N_\byx &:= -\Delta t D_{\bx_{n+1}} \boldF_\by(t_n, \bx_n, \by_n; t_{n+1}, \bx_{n+1}^k, \by_{n+1}^k), \\
        N_\byy &:= I_Y -\Delta t D_{\by_{n+1}} \boldF_\by(t_n, \bx_n, \by_n; t_{n+1}, \bx_{n+1}^k, \by_{n+1}^k),
\end{align*}
where $I_X$ and $I_Y$ denote the identity on $X$ and $Y$, respectively, and note that the blocks of course depend on the current timestep and iterate, but we have suppressed this dependence for brevity. Each iterate of Newton's method \eqref{eq:newton-method} involves inverting the block operator 
$$ N :=  \begin{pmatrix} N_\bxx & N_\bxy \\ N_\byx & N_\byy \end{pmatrix} $$
to obtain the updated iterate $(\bx_{n+1}^{k+1}, \by_{n+1}^{k+1})$. Note that $N$ is invertible, assuming sufficiently small $\Delta t$, since $N$ is a perturbation of the identity. The inverse of $N$ can be computed using the Schur complement \cite{SchurComplement}; expressing $N$ in the form
\begin{equation}\label{eq:schur-complement}
    N =  \begin{pmatrix} N_\bxx & N_\bxy \\ N_\byx & N_\byy \end{pmatrix} = \begin{pmatrix} I_X & N_\bxy N_\byy^{-1} \\ 0 & I_Y \end{pmatrix} \begin{pmatrix} N_\bx - N_\bxy N_\byy^{-1} N_\byx & 0 \\ 0 & N_\byy \end{pmatrix} \begin{pmatrix} I_X & 0 \\ N_\byy^{-1} N_\byx & I_Y \end{pmatrix},
\end{equation}
its inverse is given by
\begin{equation}\label{eq:schur-complement-inverse}
    N^{-1} = \begin{pmatrix} I_X & 0 \\ - N_\byy^{-1} N_\byx & I_Y \end{pmatrix} \begin{pmatrix} (N_\bxx - N_\bxy N_\byy^{-1} N_\byx)^{-1} & 0 \\ 0 & (N_\byy)^{-1} \end{pmatrix} \begin{pmatrix} I_X &  - N_\bxy N_\byy^{-1} \\ 0 & I_Y \end{pmatrix}.
\end{equation}
Particularly, by using the Schur complement, inverting the linear operator $N$ of size $(\dim(X) + \dim(Y))^2$ is replaced by inversions of linear operators $N_\bxx - N_\bxy N_\byy^{-1} N_\byx$ of size $\dim(X)^2$ and $N_\byy$ of size $\dim(Y)^2$. 

\textbf{Time integration of the adjoint system.}
In principle, one can integrate the adjoint equation above \eqref{eq:adjoint-coupled-evolution-equation} with any choice of (order consistent) time integration scheme, independent of the choice of integration scheme for the state dynamics. However, practical and structural considerations narrow the choice of time integration scheme.

As a practical consideration, note that the operator appearing in the adjoint equation \eqref{eq:adjoint-coupled-evolution-equation} is the adjoint of the linearization of the operator defining the state dynamics \eqref{eq:coupled-evolution-equation} (with a minus sign, but this is accounted for since the adjoint equation evolves in reverse time). In the $\mathbb{R}^N$ setting, since a matrix and its transpose have the same field of values \cite{fov}, a similar solution approach should be employed for the adjoint equation. Furthermore, as a structural consideration, since the adjoint and variational systems together satisfy the conservation law \eqref{eq:adjoint-variational-cons-coupled} in continuous time, choosing a time integration scheme for the adjoint equation that preserves this conservation law at the discrete level would provide a structure-preserving discretization of the adjoint system (for a discussion of structure-preserving methods, see for example \cite{MaWe2001, Hairer.2006}).

In \cite{TrSoLe2024}, we showed that, given a one-step time integration scheme for the state dynamics $\bu_n \mapsto \bu_{n+1}$, there exists a unique time integration scheme $\bp_{n+1} \mapsto \bp_n$ for the adjoint equation (evolving backward in time) such that a discrete analogue of the adjoint-variational quadratic conservation law holds; this resolves the structural consideration discussed above. In fact, as we will now discuss, this scheme also resolves the practical consideration.

To define this induced method, we recall the semi-implicit Euler scheme for the state dynamics \eqref{eq:coupled-semi-imp-euler}, which, for brevity, we write as
$$ \bu_{n+1} = \bu_n + \Delta t\, \boldF(t_n, \bu_n;\, t_{n+1}, \bu_{n+1}), $$
where $\bu := (\bx, \by)$. The discrete variational equation is obtained by linearizing the above scheme:
$$ \delta \bu_{n+1} = \delta \bu_n + \Delta t D_{\bu_n} \boldF(t_n, \bu_n;\, t_{n+1}, \bu_{n+1}) \delta \bu_n + \Delta t D_{\bu_{n+1}} \boldF(t_n, \bu_n;\, t_{n+1}, \bu_{n+1}) \delta \bu_{n+1}. $$
Solving for $\delta \bu_{n+1}$ in terms of $\delta \bu_n$ yields
$$ \delta \bu_{n+1} = (I - \Delta t D_{\bu_{n+1}} \boldF(t_n, \bu_n;\, t_{n+1}, \bu_{n+1}))^{-1} (I + \Delta t D_{\bu_n} \boldF(t_n, \bu_n;\, t_{n+1}, \bu_{n+1})) \delta \bu_n .$$
We then define the map $\bp_{n+1} \mapsto \bp_n$ to satisfy $\llangle \bp_n,\delta \bu_n\rrangle = \llangle \bp_{n+1}, \delta \bu_{n+1}\rrangle$. Namely, we have the following time integration scheme for the adjoint equation:
\begin{equation}\label{eq:adjoint-semi-imp-time-int}
    \bp_n =  (I + \Delta t D_{\bu_n} \boldF(t_n, \bu_n;\, t_{n+1}, \bu_{n+1}))^* (I - \Delta t D_{\bu_{n+1}} \boldF(t_n, \bu_n;\, t_{n+1}, \bu_{n+1}))^{-*} \bp_{n+1}.
\end{equation}
By construction, this map satisfies a discrete analogue of the conservation law \eqref{eq:adjoint-variational-cons-coupled}, namely,
\begin{equation}\label{eq:discrete-conservation}
    \llangle \bp_n,\delta \bu_n\rrangle = \llangle \bp_{n+1}, \delta \bu_{n+1}\rrangle,
\end{equation}
and, as shown in \cite{TrSoLe2024}, it is the unique method to do so.

\begin{remark}
    Let us compare the scheme \eqref{eq:adjoint-semi-imp-time-int} to an obvious choice for a time integration scheme of the adjoint equation, which is to apply the semi-implicit Euler scheme (in reverse time) to the adjoint equation \eqref{eq:adjoint-coupled-evolution-equation}. In this case, the scheme is
    \begin{align*}
        \bp_n = \bp_{n+1} + \Delta t [D_{\bu_{n+1}} \boldF(t_n, \bu_n;\, t_{n+1}, \bu_{n+1}))]^* \bp_n + \Delta t [D_{\bu_{n}} \boldF(t_n, \bu_n;\, t_{n+1}, \bu_{n+1}))]^* \bp_{n+1}.
    \end{align*}
    Solving for $\bp_n$ yields
    \begin{equation}\label{eq:adjoint-semi-imp-non-cons}
        \bp_n = (I - \Delta t D_{\bu_{n+1}} \boldF(t_n, \bu_n;\, t_{n+1}, \bu_{n+1}))^{-*}(I + \Delta t D_{\bu_n} \boldF(t_n, \bu_n;\, t_{n+1}, \bu_{n+1}))^* \bp_{n+1}.
    \end{equation}
    Comparing the methods \eqref{eq:adjoint-semi-imp-time-int} to \eqref{eq:adjoint-semi-imp-non-cons}, we see that the order in which the operators are applied is swapped. In fact, both methods are the same to leading order, which can be seen by simply Taylor expanding using the identity $(I - \Delta t L)^{-1} = I + \Delta t L + \mathcal{O}(\Delta t^2)$. Expanding both methods yields the same equation up to leading order $\mathcal{O}(\Delta t)$,
    \begin{align*}
        \bp_n = \bp_{n+1} & + \Delta t [D_{\bu_n} \boldF(t_n, \bu_n;\, t_{n+1}, \bu_{n+1})]^* \bp_{n+1} \\
        & + \Delta t [D_{\bu_{n+1}} \boldF(t_n, \bu_n;\, t_{n+1}, \bu_{n+1})]^* \bp_{n+1} + \mathcal{O}(\Delta t^2). 
    \end{align*}
    However, both methods are generally different (except in the trivial case that either $D_{\bu_n}\boldF = 0$ or $D_{\bu_{n+1}}\boldF = 0$, since the semi-implicit method reduces to either implicit Euler or forward Euler, respectively). Particularly, \eqref{eq:adjoint-semi-imp-time-int} satisfies the discrete conservation law \eqref{eq:discrete-conservation}, whereas \eqref{eq:adjoint-semi-imp-non-cons} does not.
\end{remark}

As we have seen, the above method \eqref{eq:adjoint-semi-imp-time-int} has resolved the structural consideration since it satisfies a discrete conservation law \eqref{eq:discrete-conservation}. Let us now turn our attention to the practical consideration. First, by assumption, since $\boldF(t_n, \bu_n;\, t_{n+1}, \bu_{n+1})$ is chosen such that its dependence on $\bu_n$ is non-stiff and its dependence on $\bu_{n+1}$ is stiff, the solution strategy for the state dynamics carries over to the solution strategy for the adjoint equation. Namely, in the method \eqref{eq:adjoint-semi-imp-time-int}, the operator being inverted is 
\begin{equation}\label{eq:adjoint-operator-to-invert}
    (I - \Delta t D_{\bu_{n+1}} \boldF(t_n, \bu_n;\, t_{n+1}, \bu_{n+1}))^*,
\end{equation}
which is stiff by assumption, whereas the operator being applied explicitly (i.e., without inversion) is
\begin{equation}
    (I + \Delta t D_{\bu_n} \boldF(t_n, \bu_n;\, t_{n+1}, \bu_{n+1}))^*,
\end{equation}
which is non-stiff by assumption.

Furthermore, let us consider inverting the operator \eqref{eq:adjoint-operator-to-invert} in block form. A direct calculation shows that the operator \eqref{eq:adjoint-operator-to-invert} is precisely the adjoint of the operator $N$ used in Newton's method for iteratively solving the state dynamics \eqref{eq:newton-method}. That is, 
$$ (I - \Delta t D_{\bu_{n+1}} \boldF(t_n, \bu_n;\, t_{n+1}, \bu_{n+1}))^* = N^* =  \begin{pmatrix} N_\bxx^* & N_\byx^* \\ N_\bxy^* & N_\byy^* \end{pmatrix}. $$
Analogous to the forward linear solve from Newton's method, the Schur complement decomposition of $N^*$ can be computed as
\begin{equation}\label{eq:schur-complement-adjoint}
    N^* =  \begin{pmatrix} N_\bxx^* & N_\byx^* \\ N_\bxy^* & N_\byy^* \end{pmatrix} = \begin{pmatrix} I_{X^*} & N_\byx^* N_\byy^{-*} \\ 0 & I_{Y^*} \end{pmatrix} \begin{pmatrix} N_\bxx^* - N_\byx^* N_\byy^{-*} N_\bxy^* & 0 \\ 0 & N_\byy^* \end{pmatrix} \begin{pmatrix} I_{X^*} & 0 \\ N_\byy^{-*} N_\bxy^* & I_{Y^*} \end{pmatrix}.
\end{equation}
Also, note that, for each timestep of the state dynamics, $N$ is inverted at each iteration of Newton's method (until convergence), whereas for the adjoint dynamics, $N^*$ is inverted once per timestep. 

\begin{remark}\label{rmk:precon}
    As another practical matter, note that if $N$ is poorly conditioned and one utilizes a preconditioner to invert the action of this operator, one would expect also to utilize a preconditioner to invert the action of $N^*$. Interpreting the preconditioning of the linear solves for the state dynamics as a transformation of the base space, the preconditioning of the linear solves for the adjoint equation provides another example of adjoint preconditioning induced from a transformation of the state dynamics, discussed in \Cref{sec:induced-from-state}.
\end{remark}

\section{Numerical example: radiation diffusion}\label{sec:numerical}
For the numerical example, we consider the radiation diffusion equations describing the isotropic transport of radiation in an optically thick medium \cite{mihalas_1984, castor_2004}, which is a system of coupled partial differential evolution equations of the form
\begin{subequations}\label{eq:radiation-diffusion}
    \begin{align}
        \frac{\partial}{\partial t} E &= \nabla \cdot ( \mathcal{D}(T) \nabla E ) - \sigma_a(T) (E- acT^4), \label{eq:radiation-diffusion-E} \\
        \rho c_v \frac{\partial}{\partial t} T &= \sigma_a(T) (E-ac T^4), \label{eq:radiation-diffusion-T}
    \end{align}
\end{subequations}
where $E$ is the radiation energy density, $T$ is the temperature, $\mathcal{D}(T)$ is the radiation diffusion coefficient which depends nonlinearly on $T$, $\sigma_a(T)$ is the absorption opacity which depends nonlinearly on $T$, $c$ is the speed of light, $a$ is the radiation constant, $\rho$ is the density, and $c_v$ is the specific heat capacity. In \eqref{eq:radiation-diffusion}, $E(t,x)$ and $T(t,x)$ are the unknowns, and the remaining data is specified.

Upon spatial semi-discretization, the system takes the form of a finite-dimensional coupled evolution equation: 
\begin{subequations}\label{eq:radiation-diffusion-semi-discrete}
    \begin{align}
        M \frac{d}{dt} \bE &= K(\bT) \bE - \Sigma_a(\bT)\bE + \Em(\bT,\bT),  \\
        M_{\rho c_v} \frac{d}{dt} \bT &= \Sigma_a(\bT)\bE - \Em(\bT,\bT).
    \end{align}
\end{subequations}
For this example, we consider an interior penalty discontinuous Galerkin discretization (see, e.g., \cite{ipdg}); the state space can be identified with $(\bE,\bT) \in \mathbb{R}^N \times \mathbb{R}^N$ using the coefficients $\bE$, $\bT$ of the finite element expansion as the coordinates. Here, $M$ is the mass matrix corresponding to \eqref{eq:radiation-diffusion-E}, $M_{\rho c_v}$ is the mass matrix with coefficient $\rho c_v$ corresponding to \eqref{eq:radiation-diffusion-T}, $K(\bT)$ is the discretization of the temperature-dependent diffusion operator $\nabla \cdot ( \mathcal{D}(T) \nabla (\ \cdot\ ))$, $ \Sigma_a(\bT)$ is a temperature-dependent bilinear form corresponding to the coefficient $\sigma_a(T)$, and $\Em(\bT_1, \bT_2)$ is a nonlinear form corresponding to $ac\sigma_a(T_1) T_2^4$; note that we have split the temperature dependence of $\Em(\bT_1, \bT_2)$ into two arguments, since we will treat the opacity dependence explicitly and the $T^4$ factor implicitly. The spatial semi-discretization for this example was performed using the finite element library MFEM \cite{mfem, mfem-web}.

The corresponding adjoint equations, utilizing mass matrices as discussed in \Cref{ex:mass-matrices}, are given by
\begin{subequations}\label{eq:adjoint-radiation-diffusion-semi-discrete}
    \begin{align}
        M \frac{d}{dt} \vec{p}_E &= -K(\bT) \vec{p}_E + \Sigma_a(\bT) \vec{p}_E - \Sigma_a(\bT) \vec{p}_T,  \\
        M_{\rho c_v} \frac{d}{dt} \vec{p}_T &= D_\bT \Em(\bT,\bT)\vec{p}_T - (D_\bT \Sigma_a(\bT) \vec{p}_T)\bE + (D_\bT \Sigma_a(\bT) \vec{p}_E) \bE \\
        & \qquad - D_\bT \Em(\bT,\bT) \vec{p}_E - (D_\bT K(\bT)\vec{p}_E)\bE. \nonumber
    \end{align}
\end{subequations}

For the state dynamics, we consider a semi-implicit scheme where the temperature dependence of the diffusion coefficient and absorption opacity are treated explicitly and the remaining quantities are treated implicitly, i.e.,
\begin{subequations}\label{eq:rad-diff-semi-imp}
    \begin{align}
        M \frac{\bE_{n+1} - \bE_n}{\Delta t} &= K(\bT_n) \bE_{n+1} -  \Sigma_a(\bT_{n})\bE_{n+1} + \Em(\bT_n,\bT_{n+1}),  \\
        M_{\rho c_v} \frac{\bT_{n+1} - \bT_n}{\Delta t} &=  \Sigma_a(\bT_{n})\bE_{n+1} - \Em(\bT_n,\bT_{n+1}),
    \end{align}
\end{subequations}
which is the standard approach used in deterministic radiation simulation for many years \cite{Larsen.1988}. This nonlinear system is solved using Newton's method and the Schur complement, whereas the adjoint equation \eqref{eq:adjoint-radiation-diffusion-semi-discrete} is solved using the induced time integration method, as described in \Cref{sec:time-integration}.
All of the experiments performed use the preconditioned conjugate gradient method \cite{pcg} with an algebraic multigrid (AMG) preconditioner \cite{amg} to precondition the linear solves in both the forward and backward runs, as discussed in \Cref{rmk:precon}. When run without a preconditioner, the forward and backward runs did not converge within a reasonable runtime due to the stiffness of the nonlinear diffusion, with the diffusion coefficient $\mathcal{D}(T)$ varying on the order of $\max(T)^3/\min(T)^3 = 1200^3/0.025^3 \approx 1.1 \times 10^{14}$ over the domain.

For this example, we consider the following optimization problem, given by an inverse problem of a perturbation of a thick Marshak wave. Here, the spatial domain is $[0,l] \ni x$ where $l = \SI{0.25}{\cm}$, and the time interval is $[0, t_f]$ where $t_f = \SI{1e-8}{\s}$. We use a timestep $\Delta t = \SI{5e-13}{\s}$, with $N = 100$ spatial degrees of freedom for both $\bE$ and $\bT$. For the remaining problem data, we take $c = \SI{2.99792e10}{\cm\per\s}$, $a = \SI{137.2}{\erg\per\cm\cubed\per\eV\tothe{4}}$, $\rho = \SI{1}{\g\per\cm\cubed}$, $c_v = \SI{3e12}{\erg\per\cm\cubed\per\eV}$, $\sigma_a(T) = 10^{12}/T^3\,\si{\per\cm}$, and $\mathcal{D}(T) = (c/3)/\sigma_a(T)\, \si{\cm\squared\per\s}$.

First, consider the dynamics of a thick Marshak wave propagating over a constant blackbody equilibrium state with initial condition $T(0,x) = 0.025$ and $E(0,x) = acT(0,x)^4 \approx 1.61 \times 10^6$ for all $x$. The Marshak wave is driven by imposing a surface temperature $T(t,0) = 1200$ at $x=0$ for all $t>0$; for a more detailed discussion of the dynamics of the Marshak wave, see, e.g., \cite{marshak}. Snapshots of the evolution of the thick Marshak wave over the constant equilibrium initial state are shown in \Cref{fig:const_ET}.
\begin{figure}[H]
\begin{center}
\includegraphics[width=160mm]{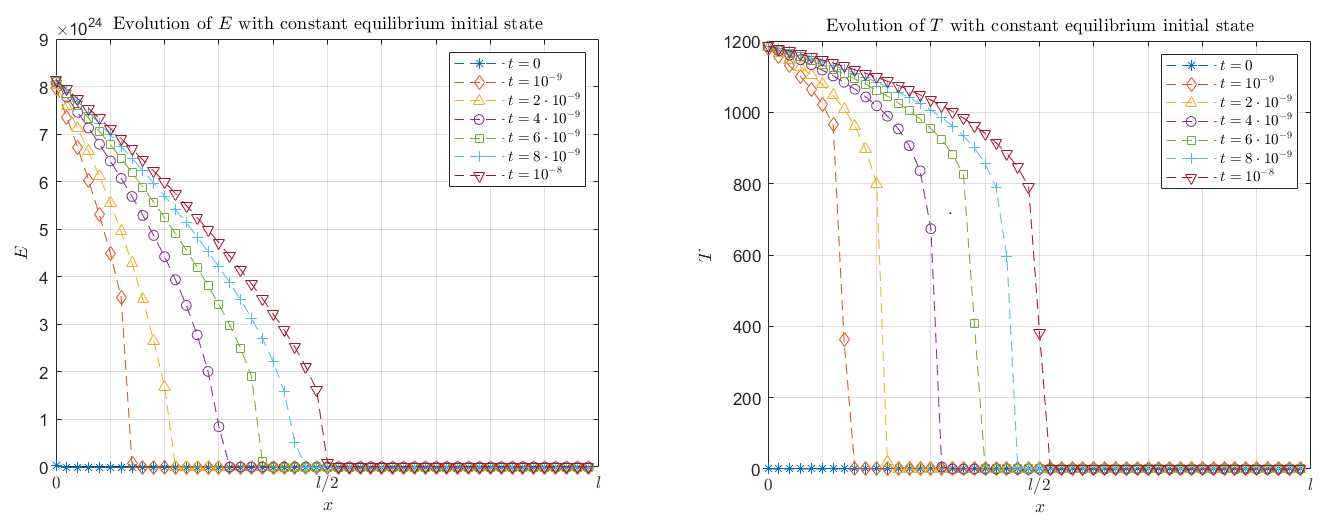}
\caption{Snapshots of the evolution of $E$ (left) and $T$ (right) of the thick Marshak wave with a constant equilibrium initial state.}
\label{fig:const_ET}
\end{center}
\end{figure}

Now, we consider the following inverse problem, which is to reconstruct the initial conditions $E(0)$ and $T(0)$ such that the radiation energy and temperature at the final time, $E(t_f)$ and $T(t_f)$, match some observed state at the final time, $E^*$ and $T^*$, subject to the dynamics of the thick Marshak wave. We will further impose the constraint that the initial state $E(0)$ and $T(0)$ are in local equilibrium, i.e., $E(0,x) = ac T(0,x)^4$ for all $x \in (0,l)$ (this is local in the sense that $E(0)$ and $T(0)$ are still allowed to vary in space, unlike the initial state shown in \Cref{fig:const_ET}). This can be expressed as a constrained optimization problem of the form 
\begin{align}\label{eq:opt-problem-ET}
    \min_{E_0, T_0}\,& C(E(t_f),T(t_f)) := \underbrace{\frac{1}{2} \| E(t_f) - E^*\|_{L^2}^2}_{=:\ C_E} + \underbrace{\frac{1}{2} \| T(t_f) - T^*\|_{L^2}^2}_{=:\ C_T}, \\
    &\text{such that the state dynamics are given by } \eqref{eq:radiation-diffusion}, \nonumber \\
    &\text{with } E(0,x) = E_0(x), T(0,x) = T_0(x) \text{ satisfying } E_0(x) = acT_0(x)^4 \text{ for all } x \in (0,l). \nonumber
\end{align}
For the observed values $E^*$ and $T^*$ at the final time, we consider a perturbation of the thick Marshak wave at the final time. This is shown in \Cref{fig:perturbed_ET}.
\begin{figure}[H]
\begin{center}
\includegraphics[width=160mm]{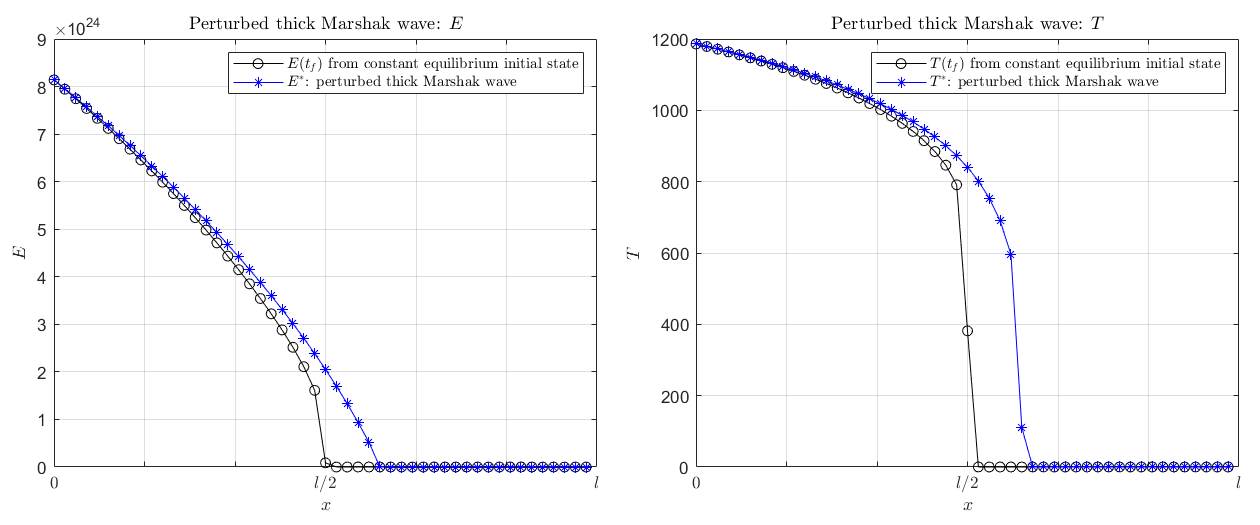}
\caption{The observed values $E^*$ compared to $E(t_f)$ produced from the constant equilibrium initial conditions (left) and $T^*$ compared to $T(t_f)$ produced from the constant equilibrium initial conditions (right).}
\label{fig:perturbed_ET}
\end{center}
\end{figure}

We solve this optimization problem using an adjoint-based gradient approach, as discussed in \Cref{sec:duality}, modified appropriately for the local equilibrium constraint. To handle the local equilibrium constraint for the initial conditions, we use the projected gradient descent method (see, e.g., \cite{projGD}). Furthermore, we will consider two different projections: the orthogonal projection onto the constraint and a coordinate projection onto the constraint. 

Define the constraint function on the space of initial conditions, $\Phi: \mathbb{R}^N \times \mathbb{R}^N \rightarrow \mathbb{R}^N$, whose zero level set defines the constraint set. Its $i^{th}$ component is given by 
\begin{equation}\label{eq:constraint-components}
    \Phi_i(\bE_0, \bT_0) = (\bE_0)_i - ac(\bT_0)_i^4.
\end{equation}
At the current iterate $(\bE^k_0, \bT^k_0)$, the projected gradient descent method begins with an unconstrained gradient step
\begin{subequations}\label{eq:unconstrained-grad-step}
    \begin{align}
    \bE^{\text{un}}_0 &:= \bE^{k}_0 - \gamma p_{\bE}^k(0), \\
    \bT^{\text{un}}_0 &:= \bT^{k}_0 - \gamma p_{\bT}^k(0), 
\end{align}
\end{subequations}
followed by a projection onto the constraint \eqref{eq:constraint-components}
\begin{subequations}\label{eq:oblique-projection-step}
\begin{align}
    \bE^{k+1}_0 &= \bE^{\text{un}}_0 - \text{diag}\big(\vec{\lambda}\big) \vec{\eta}_{\bE}(\bE^{\text{un}}_0, \bT^{\text{un}}_0), \\
    \bT^{k+1}_0 &= \bT^{\text{un}}_0 - \text{diag}\big(\vec{\lambda}\big) \vec{\eta}_{\bT} (\bE^{\text{un}}_0, \bT^{\text{un}}_0),  
\end{align}
\end{subequations}
where $\vec{\lambda} \in \mathbb{R}^N$ is a Lagrange multiplier chosen to satisfy the constraint $\Phi(\bE^{k+1}_0, \bT^{k+1}_0) = 0$ and $\vec{\eta} \coloneqq (\vec{\eta}_{\bE}(\bE^{\text{un}}_0, \bT^{\text{un}}_0), \vec{\eta}_{\bT}(\bE^{\text{un}}_0, \bT^{\text{un}}_0)) \in \mathbb{R}^N \times \mathbb{R}^N$ is an \emph{oblique} vector, i.e., a vector which has a non-zero component parallel to the derivative of the constraints
$$\begin{pmatrix} D_{\bE_0}\Phi(\bE_0, \bT_0) \\ D_{\bT_0}\Phi(\bE_0, \bT_0) \end{pmatrix}. $$
Here, the derivatives $D_{\vec{E}_0} \Phi$ and $ D_{\vec{T}_0} \Phi$ are interpreted component-wise, i.e., the $i^{th}$ component of $D_{\vec{E}_0} \Phi$ is $D_{(\vec{E}_0)_i} \Phi_i$ and the $i^{th}$ component of $D_{\vec{T}_0} \Phi$ is $D_{(\vec{T}_0)_i} \Phi_i$; this simplification arises since the constraint \eqref{eq:constraint-components} is local (the $i^{th}$ component $\Phi_i$ only depends on $(\vec{E}_0)_i$ and $(\vec{T}_0)_i$). 

For the oblique vector, one choice is to use the derivative of the constraints itself, in which case the projection is the orthogonal projection onto the constraints. The orthogonal projection step is then 
\begin{subequations}\label{eq:orthogonal-projection-step}
\begin{align}
    \bE^{k+1}_0 &= \bE^{\text{un}}_0 - \text{diag}\big(\vec{\lambda}\big) D_{\bE_0}\Phi(\bE^{\text{un}}_0, \bT^{\text{un}}_0), \\
    \bT^{k+1}_0 &= \bT^{\text{un}}_0 - \text{diag}\big(\vec{\lambda}\big) D_{\bT_0}\Phi(\bE^{\text{un}}_0, \bT^{\text{un}}_0).
\end{align}
\end{subequations}
With this choice of oblique vector used in \eqref{eq:oblique-projection-step}, since the constraint is local, each component of the constraint $\Phi_i(\bE^{k+1}_0, \bT^{k+1}_0) = 0$ can be solved for the corresponding component $\lambda_i$ of the Lagrange multiplier separately. 

Alternatively, we also consider a choice of an oblique vector given by $\vec{\eta} = (\vec{\eta}_{\bE}, 0)$. Since the constraint \eqref{eq:constraint-components} can be solved for $\vec{E}_0$ in terms of $\vec{T}_0$, this projection step can be done simply as
\begin{subequations}\label{eq:E-coordinate-projection-step}
\begin{align}
    \bT^{k+1}_0 &= \bT^{\text{un}}_0, \\
    \bE^{k+1}_0 &= \bE^{\text{un}}_0 - \text{diag}\big(\vec{\lambda}\big) \vec{\eta}_{\bT} (\bE^{\text{un}}_0, \bT^{\text{un}}_0) = a\left(\bT^{k+1}_0\right)^4, 
\end{align}
\end{subequations}
where the last equality is interpreted component-wise. This method has the simple property that one does not need to specify $\vec{\eta}_{\bE}$ nor solve for the Lagrange multiplier $\vec{\lambda}$, since they are implicitly defined by solving for $\bE^{k+1}_0$ in terms of $\bT^{k+1}_0$. We refer to this as the $\vec{E}$-coordinate oblique projection, since the projection onto the constraint \eqref{eq:E-coordinate-projection-step} is given by projecting in the $\vec{E}$-coordinate only.

In the unconstrained gradient step \eqref{eq:unconstrained-grad-step}, we use a constant stepsize $\gamma = 0.1$ across all runs for both projection methods. Due to the large scale-separation between the energy $E(t_f) \sim \mathcal{O}(10^{24})$ and temperature $T(t_f) \sim \mathcal{O}(10^3)$, we use the block-diagonal scale preconditioner described in \Cref{sec:scale-preconditoning-coupled}. To normalize scales, we expect preconditioning using the relative magnitude, $\max E(T_f) / T(T_f) \sim 10^{21}$ to be appropriate. In \Cref{fig:costETorthog} and \Cref{fig:costEToblique}, we compare various choices of values for the scale preconditioner centered around this target value, for the orthogonal projection and $\vec{E}$-coordinate oblique projection, respectively. Note that there is virtually no added computational cost in the adjoint preconditioning used here, since we are applying a state-independent and furthermore, diagonal preconditioner.
\begin{figure}[H]
\begin{center}
\includegraphics[width=160mm]{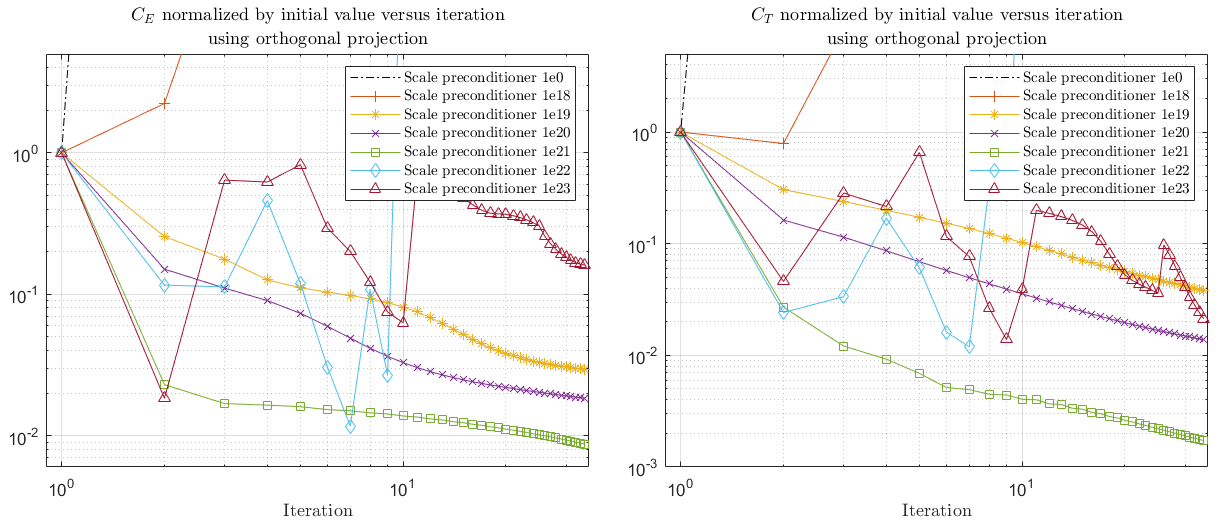}
\caption{$E$ component of the cost function (left) and $T$ component of the cost function (right) versus iteration, for several values of the scale preconditioner, using the orthogonal projection.}
\label{fig:costETorthog}
\end{center}
\end{figure}

\begin{figure}[H]
\begin{center}
\includegraphics[width=160mm]{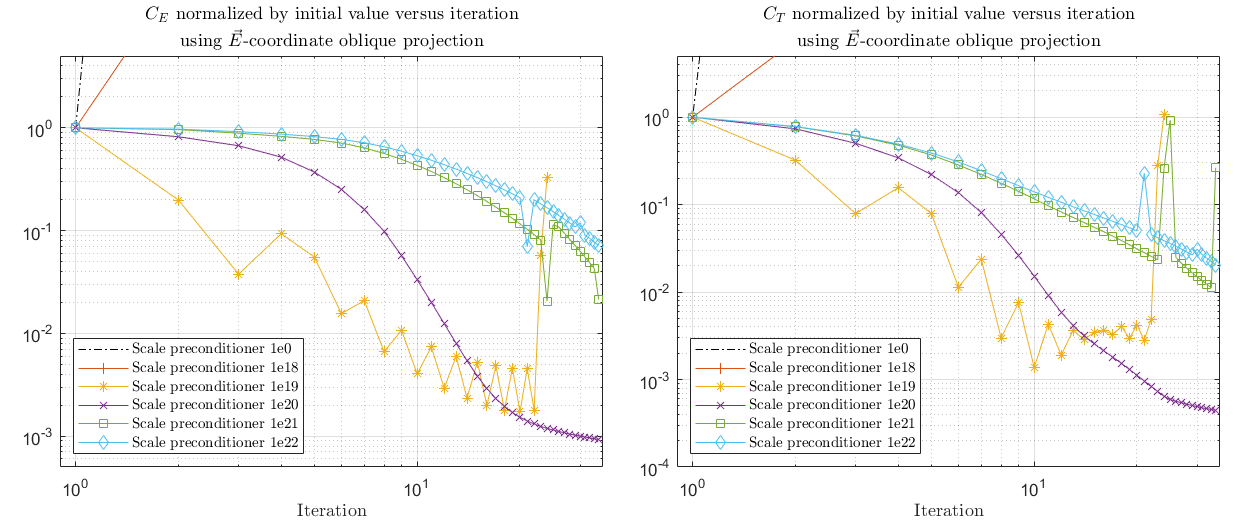}
\caption{$E$ component of the cost function (left) and $T$ component of the cost function (right) versus iteration, for several values of the scale preconditioner, using the $\vec{E}$-coordinate oblique projection.}
\label{fig:costEToblique}
\end{center}
\end{figure}
As can be seen in \Cref{fig:costETorthog} and \Cref{fig:costEToblique}, by utilizing a suitable scale preconditioner based around relative coupled variable magnitudes, the gradient-based optimization exhibits rapid, stable convergence in $\mathcal{O}(10)$ iterations. In practice, we also observed improved stability in both the forward and backward runs with a suitable scale preconditioner. When utilizing the naive gradient descent (i.e., scale preconditioner $1$), the cost immediately diverges at the first iterate; we also ran the naive gradient descent with a stepsize $\delta \sim 1e^{-21}$, sufficiently small for stability, but observed essentially no decrease in the cost function due to the improper scaling of the components of the adjoint equation discussed in \Cref{ex:scale-preconditioner-block}.

\begin{remark}\label{rmk:instability}
    Note that the convergence shown in \Cref{fig:costETorthog} and \Cref{fig:costEToblique} appear sensitive to the scale preconditioner, exhibiting instabilities for the non-optimal values of the scale preconditioner. Without a rigorous proof, we suspect that this is possibly due to two facts: (i) as discussed in \Cref{sec:discussion}, since the state dynamics are nonlinear, the optimization problem may be non-convex and thus, instabilities typical to using gradient descent for non-convex optimization may arise and (2) we observed empirically that the behavior of the radiation-diffusion system is very sensitive to initial conditions, due to its stiff nonlinear dynamics; as such, utilizing gradient descent with the non-optimal values of the scale preconditioner may lead to unbalanced search directions in the space of initial conditions to which the system is highly sensitive. To further support that the observed instabilities are inherent to the stiff nonlinear dynamics of the radiation diffusion system and not the proposed approach, we perform an ablation study in \Cref{sec:ablation-study} where we examine the effect of scale preconditioning on coupled linear diffusion equations.
\end{remark}

The reconstructed initial conditions from the lowest cost function values for the orthogonal projection and $\vec{E}$-coordinate oblique projection are shown in \Cref{fig:recon_ET}, along with the initial values used to generate the target states $E^*$ and $T^*$.
\begin{figure}[H]
\begin{center}
\includegraphics[width=160mm]{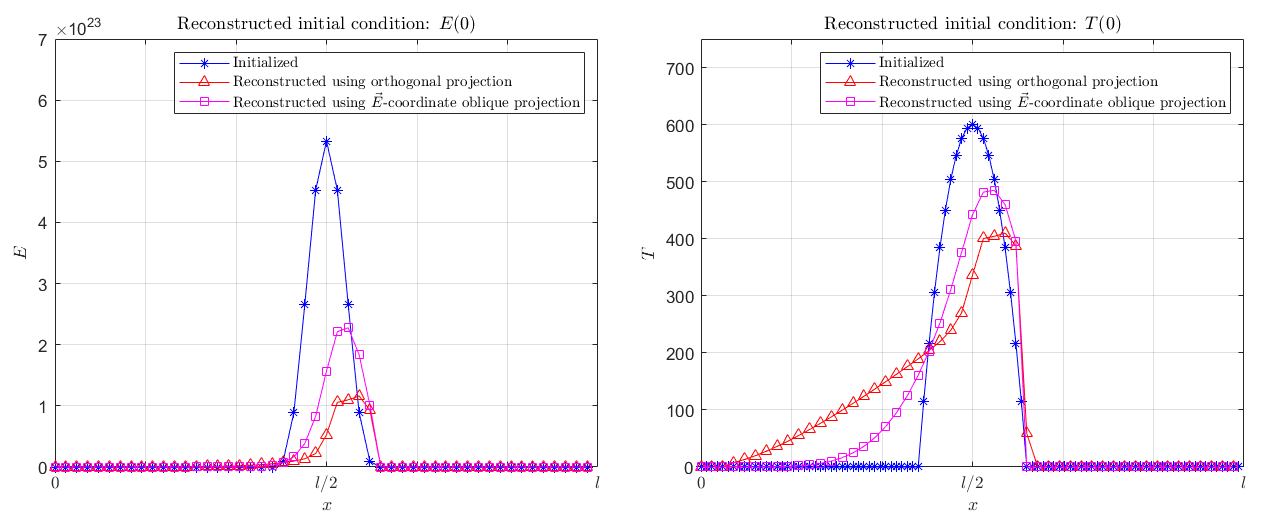}
\caption{The initialized and reconstructed initial conditions, $E(0)$ (left) and $T(0)$ (right).}
\label{fig:recon_ET}
\end{center}
\end{figure}

Finally, a comparison of the final states from the constant equilibrium initial conditions (i.e., the unperturbed thick Marshak wave), the observed final states, and the reconstructed final states using the orthogonal projection and $\vec{E}$-coordinate oblique projection are shown in \Cref{fig:comp_uor_ET}.

\begin{figure}[H]
\begin{center}
\includegraphics[width=160mm]{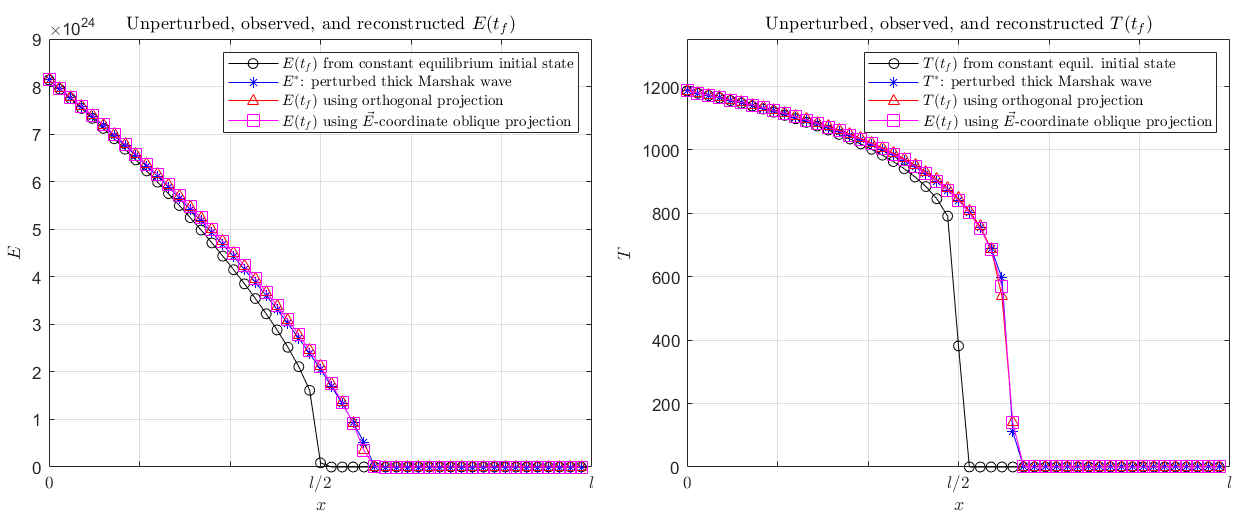}
\caption{Comparison of the unperturbed, observed and reconstructed values at the final time for $E$ (left) and $T$ (right).}
\label{fig:comp_uor_ET}
\end{center}
\end{figure}
\begin{remark}
    It is interesting to note that the initial states, shown in \Cref{fig:recon_ET}, are qualitatively much different between the initialized, reconstructed using orthogonal projection, and reconstructed using the $\vec{E}$-coordinate oblique projection. On the other hand, the terminal states are in close agreement, as shown in \Cref{fig:comp_uor_ET}. This can be understood from the ill-posedness of the inverse problem, in that multiple initial states can evolve toward the same terminal state, particularly due to the nonlinear diffusion appearing in \Cref{eq:radiation-diffusion}.
\end{remark}

\subsection{Ablation study}\label{sec:ablation-study}
As discussed in \Cref{rmk:instability}, to provide further support that the instabilities observed in the convergence for non-optimal scale preconditioner values are due to the inherent difficulty of the radiation diffusion state dynamics, and not the proposed preconditioning approach, we perform here an ablation study which examines the effect of scale preconditioning on coupled linear diffusion equations which exhibit scale separation. As we will see, such instabilities do not appear in this simplified setting. 

Consider coupled diffusion equations on the spatial domain $[0,1]$ with periodic boundary conditions,
\begin{equation}
    \partial_t \begin{pmatrix} u_1 \\ u_2 \end{pmatrix} = \begin{pmatrix}
        \partial_{xx} - 1 & \alpha \\ \beta & \partial_{xx} - 1
    \end{pmatrix} \begin{pmatrix} u_1 \\ u_2 \end{pmatrix},
\end{equation}
where $\alpha \gg 1$ and $\beta \sim \mathcal{O}(1/\alpha)$, with initial condition $\bu(0) = \bu_0$, where $\bu := (u_1,\ u_2)^T$. We consider again an optimization problem of minimizing a terminal cost of the form $C(\bu(T)) = \frac{1}{2} \|\bu(T) - \bu^*\|^2$ over the space of initial conditions. Expanding the solution into Fourier modes, the $k^{th}$ Fourier mode decouples into a $2 \times 2$ ODE with evolution matrix
$$ L_k = \begin{pmatrix}
    -k^2\pi^2 - 1 & \alpha \\ \beta & -k^2\pi^2 - 1
\end{pmatrix} = L - k^2\pi^2 I, \quad L := \begin{pmatrix}
    -1 & \alpha \\ \beta & -1
\end{pmatrix}. $$
The time $0$ to time $T$ flow map for the $k^{th}$ eigenmode is given by $(\Phi_k)_0^T = e^{-k^2\pi^2 T}\Phi_0^T$ where $(\Phi)_0^T := e^{LT}$. As discussed in \Cref{sec:discussion}, we can understand the effect of the adjoint preconditioning by examining the conditioning of the reduced cost $\tilde{C}(\bu_0)$ given by pulling back the cost through the flow map. Since the eigenmodes decouple and the dependence on $k$ only appears as a scalar factor in the flow map, the condition number of the Hessian $H_k = ((\Phi_k)_0^T)^T (\Phi_k)_0^T $ associated to this optimization problem is independent of $k$. In particular, it suffices to consider any particular $k$, so we take $k=0$, for which we denote the Hessian $H := H_0$. Using the scale preconditioner $P$ of \Cref{sec:scale-preconditoning-coupled}, we numerically compute the condition number for the preconditioned Hessian $\kappa(P^{-1/2}HP^{-1/2})$ with varying $s$, particularly noting the minimizer $s^*$ and the value $s^* = \alpha/\beta \approx s^*$. We demonstrate this numerically with $k=0$, $\alpha = 100$, $\beta = 0.01$, $T=1$ with $\bu^* = \Phi_0^T \bu_0^*$ and $\bu_0^* = (50,3)^T$ using adjoint-based gradient descent for the $2\times 2$ system with initial iterate $\bu_0 = (0,0)$, with fixed gradient descent stepsize $\gamma = 1$. 

\begin{figure}[H]
\begin{center}
\includegraphics[width=165mm]{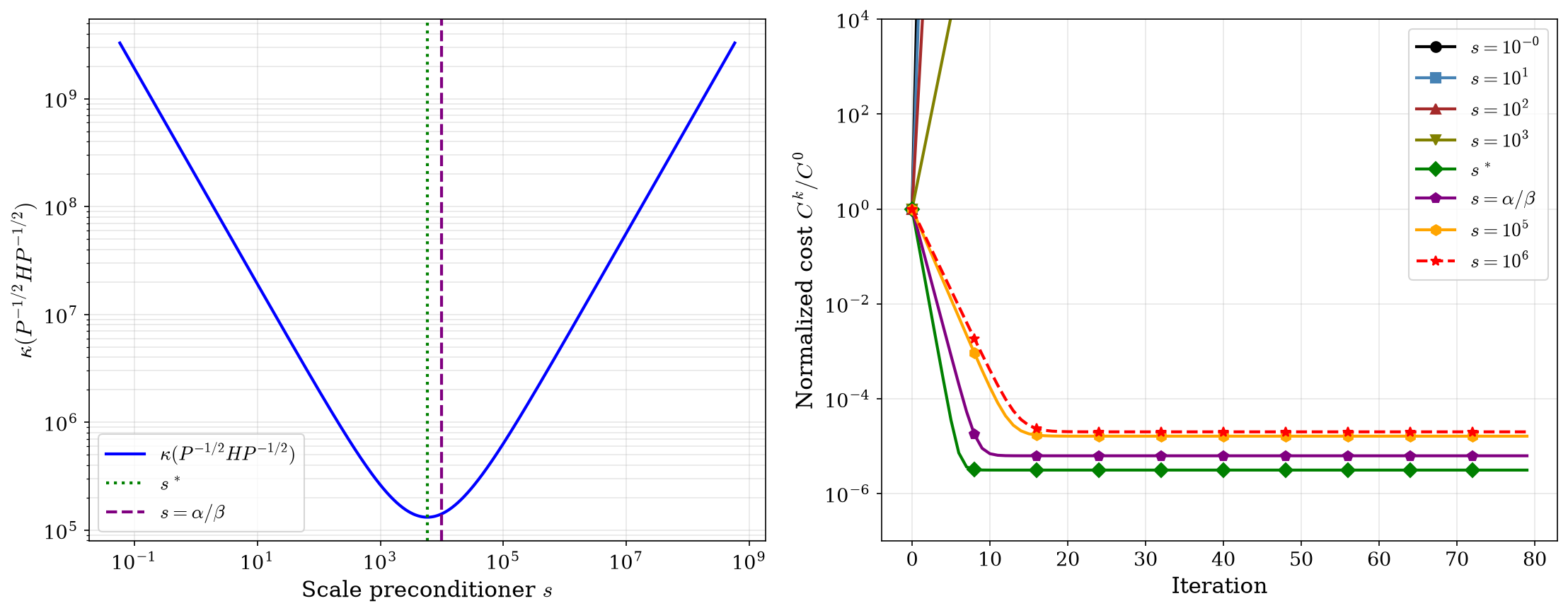}
\caption{(Left) Condition number as a function of scale preconditioner. (Right) Normalized cost versus iteration.}
\label{fig:ablation}
\end{center}
\end{figure}

As shown in the left plot of \Cref{fig:ablation}, the condition number is minimized at $s^*$ and grows away from this point. The right plot shows convergence for $s^*$ and $s = \alpha/\beta$, whereas for values of the scale preconditioner away from the optima, the convergence either stalls or diverges, but does not exhibit the unstable behavior discussed in \Cref{rmk:instability}. By removing the nonlinear features of the radiation diffusion while keeping the characteristic scale separation of the coupled dynamics, the ablation study supports the explanation of the sensitivity and instability of the radiation diffusion example provided in \Cref{rmk:instability}.

\section{Conclusion}
In this paper, we introduced preconditioning transformations of adjoint systems for evolution equations inspired by the preconditioning of gradient descent in unconstrained optimization. We showed that these transformations can be naturally interpreted as symplectomorphisms of the canonical adjoint system and hence, preserve the property the adjoint equation backpropagates the derivative of an objective function. As an application, we discussed adjoint systems for coupled evolution equations and particularly discussed scale preconditioning of the adjoint equation for evolution equations that exhibit large scale-separation. An example of this scale preconditioning was used in the numerical example, where we considered an inverse problem for the radiation diffusion equations, which is a coupled evolution equation exhibiting large scale-separation.

For future work, it would be interesting to utilize the adjoint preconditioning induced from a state-dependent duality pairing to develop preconditioned adjoint methods that are adapted to the state dynamics; this would be particularly interesting for state dynamics where the features of the solution vary drastically over time. It would also be interesting to develop adjoint preconditioners for spatially multiscale problems. Furthermore, we plan to explore applications of the adjoint preconditioning induced from transformation of the state dynamics, particularly when these transformations are nonlinear, e.g., for systems with state-dependent mass matrices. Finally, since many state of the art machine learning training algorithms can be seen as preconditioned gradient descent, e.g., \cite{yao2021adahessian,gupta2018shampoo,martens2015optimizing}, and backpropagation of a neural differential equation can be related to the discretization of the adjoint equation \cite{MaMiYa2023}, it would be interesting to explore connections and applications of the adjoint preconditioning framework in the context of training neural networks.

\section*{Acknowledgements} The authors would like to thank the referees for their careful review and for their valuable comments and suggestions. Los Alamos National Laboratory Report LA-UR-25-24953.

\section*{Funding}
BKT was supported by the Marc Kac Postdoctoral Fellowship at the Center for Nonlinear Studies at Los Alamos National Laboratory. BSS was supported by the DOE Office of Advanced Scientific Computing Research Applied Mathematics program through Contract No. 89233218CNA000001. HFB was supported by the Los Alamos National Laboratory Summer Student Program 2024, by the Laboratory Directed Research and Development program of Los Alamos National Laboratory under project number 20230068DR, and by the Department of Defense through the National Defense Science \& Engineering Graduate Fellowship Program. 

\section*{Author Contributions}
BKT and BSS led the conceptualization, drafting, planning, and organization of the manuscript. All of the authors contributed to the theoretical developments of the manuscript and the numerical example. All of the authors have reviewed the manuscript and approved it for submission.

\section*{Data Availability Statement}
The data generated in this study is available upon reasonable request from the corresponding author.

\appendix


\bibliographystyle{plainnat}
\bibliography{adjointprec_final.bib}

\end{document}